\documentclass[a4paper,reqno,11pt]{amsart}
\usepackage[tmargin=30mm,bmargin=30mm,hmargin=25mm]{geometry}
\usepackage{amssymb}
\usepackage{amstext}
\usepackage{amsmath}
\usepackage{amscd}
\usepackage{amsthm}
\usepackage{amsfonts}
\usepackage{enumerate}
\usepackage{latexsym}
\usepackage{mathrsfs}
\input xy
\xyoption{all}
\usepackage{pstricks}

\newtheorem{theorem}{Theorem}[section]
\newtheorem{corollary}[theorem]{Corollary}
\newtheorem{lemma}[theorem]{Lemma}
\newtheorem{proposition}[theorem]{Proposition}
\theoremstyle{definition}
\newtheorem{definition}[theorem]{Definition}
\newtheorem{remark}[theorem]{Remark}
\newtheorem{example}[theorem]{Example}
\numberwithin{equation}{section}
\newcommand{\add}{\mathsf{add}\hspace{.01in}}

\renewcommand{\mod}{\mathsf{mod}\hspace{.01in}}

\newcommand{\cok}{\operatorname{Cok}\nolimits}
\newcommand{\cocone}{\operatorname{Cocone}\nolimits}
\newcommand{\cone}{\operatorname{Cone}\nolimits}
\newcommand{\End}{\operatorname{End}\nolimits}
\newcommand{\Ext}{\operatorname{Ext}\nolimits}
\newcommand{\Hom}{\operatorname{Hom}\nolimits}
\newcommand{\id}{\operatorname{id}\nolimits}
\renewcommand{\ker}{\operatorname{Ker}\nolimits}
\newcommand{\op}{\operatorname{op}\nolimits}
\newcommand{\pd}{\operatorname{pd}\nolimits}
\newcommand{\rad}{\operatorname{rad}\nolimits}
\newcommand{\soc}{\operatorname{soc}\nolimits}
\renewcommand{\top}{\operatorname{top}\nolimits}

\begin{document}
\title[Mixed standardization and Ringel duality]{Mixed standardization and Ringel duality}
\author{Takahide Adachi}\address{T.~Adachi: Faculty of Global and Science Studies, Yamaguchi University, 1677-1 Yoshida, Yamaguchi 753-8541, Japan}\email{tadachi@yamaguchi-u.ac.jp}\thanks{T.~Adachi is supported by JSPS KAKENHI Grant Number JP20K14291.}
\author{Mayu Tsukamoto}\address{M.~Tsukamoto: Graduate school of Sciences and Technology for Innovation, Yamaguchi University, 1677-1 Yoshida, Yamaguchi 753-8512, Japan}\email{tsukamot@yamaguchi-u.ac.jp}\thanks{M.~Tsukamoto is supported by JSPS KAKENHI Grant Number JP19K14513.}

\begin{abstract}
Dlab--Ringel's standardization method gives a realization of a standardly stratified algebra.
In this paper, we construct mixed stratified algebras, which are a generalization of standardly stratified algebras, following Dlab--Ringel's standardization method.
Moreover, we study a Ringel duality of mixed stratified algebras from the viewpoint of stratifying systems.
\end{abstract}
\maketitle

\section{Introduction}
The notion of quasi-hereditary algebras was introduced by Cline, Parshall and Scott \cite{CPS88} and has its origin in the representation theories of complex Lie algebras and algebraic groups.
From the viewpoint of the representation theory of finite dimensional algebras, Dlab and Ringel intensively studied quasi-hereditary algebras (for example \cite{DR89, DR92, R91, D96}).
Moreover, Iyama \cite{I03} proved that each finite dimensional algebra has a finite representation dimension by using a property of quasi-hereditary algebras.
In the theory of quasi-hereditary algebras, standard modules play a crucial role.
By focusing on standard modules, standardly stratified algebras are defined as a natural generalization of quasi-hereditary algebras \cite{CPS96, D96}.

Dlab and Ringel \cite{DR92} introduced the notion of a standardizable set, which behaves like a set of standard modules, and gave a realization of a quasi-hereditary algebra as the endomorphism algebra of an (Ext-)projective object of the smallest extension-closed subcategory $\mathcal{F}(\Theta)$ containing a standardizable set $\Theta$.
Moreover, Erdmann and S\'aenz \cite{ES03} generalized this result to  standardly stratified algebras.
Let $\Theta:=(\Theta(1), \Theta(2), \ldots, \Theta(t))$ be an ordered set of objects in an abelian category $\mathcal{A}$.
Then $\Theta$ is called a \emph{standardizable set} if it satisfies that (S1) $\Theta(i)$ is a stone for each $i \in [1,t]$, (S2)  $\mathcal{A}(\Theta(i), \Theta(j))=0$ for each $i>j$ and (S3)  $\Ext_{\mathcal{A}}^{1}(\Theta(i), \Theta(j))=0$ for each $i>j$.

\begin{theorem}[Dlab--Ringel's standardization method] \label{intro_thm1}
Let $\mathcal{A}$ be a Krull--Schmidt Ext-finite abelian category.
Let $\Theta:=(\Theta(1), \Theta(2), \ldots, \Theta(t))$ be an ordered set of objects in $\mathcal{A}$.
Then the following statements hold.
\begin{itemize}
\item[(1)] If $\Theta$ is a standardizable set, then there exists a projective object $\mathbb{P}:=(P(1), P(2), \ldots, P(t))$ in $\mathcal{F}(\Theta)$ such that the pair $(\Theta, \mathbb{P})$ is a stratifying system in $\mathcal{A}$, that is, $\Theta$ satisfies (S1) and (S2), and there exists an exact sequence
\begin{align}
0 \to K(i) \to P(i)\to \Theta(i) \to 0  \notag
\end{align}
such that $K(i)\in\mathcal{F}(\Theta(i+1),\Theta(i+2),\ldots,\Theta(t))$ and $P(i)$ is an indecomposable projective object in  $\mathcal{F}(\Theta)$ for each $i\in[1,t]$.
\item[(2)] Assume that $\mathcal{A}$ is Hom-finite.
Let $(\Theta,\mathbb{P}')$ be a stratifying system in $\mathcal{A}$ and $\Psi:=\mathcal{A}(\mathbb{P}', -)$.
Then the endomorphism algebra $\End_{\mathcal{A}}(\mathbb{P}')$ is a standardly stratified algebra with respect to standard modules $(\Psi (\Theta(1)), \Psi(\Theta(2)), \ldots, \Psi(\Theta(t)))$ and $\mathcal{F}(\Theta)$ is equivalent to $\mathcal{F}(\Psi(\Theta))$.
\end{itemize}
\end{theorem}

Subsequently, Dlab--Ringel's standardization method was studied in a triangulated category \cite{MS16} and an exact category \cite{S19}.
As an analog of standardizable sets, Mendoza, Platzeck and Verdecchia \cite{MPV14} introduced a ``proper'' standardizable set, which behaves like a set of proper standard modules, that is, the condition (S1) is replaced with the condition that $\Theta(i)$ is a brick.
Under a certain condition, the counterpart of Theorem \ref{intro_thm1} holds.

As a further generalization of standardly stratified algebras, \'Agoston, Dlab and Luk\'acs \cite{ADL98}  introduced the notion of stratified algebras of type $\mathbf{s}$.
Moreover, this is a special class of $\varepsilon$-stratified algebras \cite{BS}.
To distinguish with some classes of stratified algebras, we call a stratified algebra of type $\mathbf{s}$ a \emph{mixed stratified algebra} in this paper.
Let $A$ be a finite dimensional algebra and $(e_{1}, e_{2}, \ldots, e_{n})$ the complete ordered set of primitive orthogonal idempotents.
Let $\Delta(i)$ (respectively, $\overline{\Delta}(i)$) denote the $i$-th standard (respectively, proper standard) module.
An algebra $A$ is called a \emph{mixed stratified algebra} if $A\in \mathcal{F}(\Xi(1),\Xi(2),\ldots, \Xi(n))$, where $\Xi(i) \in\{\Delta(i), \overline{\Delta}(i)\}$ for each $i\in[1,n]$.

Our aim of this paper is to study Dlab--Ringel's standardization method and Ringel duality for mixed stratified algebras.
We introduce the notion of mixed standardizable sets in an extriangulated category, which was introduced in \cite{NP19} as a simultaneous generalization of a triangulated category and an exact category.
Let $\mathcal{C}:=(\mathcal{C},\mathbb{E},\mathfrak{s}, \mathbb{E}^{-1})$ be an extriangulated category with a negative first extension.
Let $\Theta:=(\Theta(1), \Theta(2), \ldots, \Theta(t))$ be an ordered set of objects in $\mathcal{C}$.
Then $\Theta$ is called a \emph{mixed standardizable set} if it satisfies that (MS1) $\Theta(i)$ is a brick or a stone for each $i\in[1,t]$, (MS2) $\mathcal{C}(\Theta(i), \Theta(j))=0$ for each $i>j$, (MS3) $\mathbb{E}(\Theta(i), \Theta(j))=0$ for each $i>j$ and (MS4) $\mathbb{E}^{-1}(\Theta,\Theta)=0$.
Here, an object $M$ is called a \emph{stone} if it is indecomposable and $\mathbb{E}(M,M)=0$.
As a generalization of Theorem \ref{intro_thm1}, we have the following result.

\begin{theorem}[Theorems \ref{mainthm1} and \ref{mainthm2}]
Let $\mathcal{C}$ be a Krull--Schmidt $\mathbb{E}$-finite   extriangulated category with a negative first extension.
Let $\Theta:=(\Theta(1), \Theta(2), \ldots, \Theta(t))$ be an ordered set of objects in $\mathcal{C}$.
Then the following statements hold.
\begin{itemize}
\item[(1)] If $\Theta$ is a mixed standardizable set satisfying a certain finiteness condition in $\mathcal{F}(\Theta)$ (see Definition \ref{def_ms}(2)), then there exists a projective object $\mathbb{P}=(P(1),P(2),\ldots, P(t))$ in  $\mathcal{F}(\Theta)$ such that the pair $(\Theta, \mathbb{P})$ is a mixed stratifying system in $\mathcal{C}$, that is, $\Theta$ satisfies (MS1), (MS2) and (MS4), and there exists an $\mathfrak{s}$-conflation
\begin{align}
K(i) \to P(i) \to \Theta(i) \dashrightarrow \notag
\end{align}
such that $K(i) \in \mathcal{F}(\Theta(i),\Theta(i+1),\ldots,\Theta(t))$ and $P(i)$ is an indecomposable projective object in $\mathcal{F}(\Theta)$ for each $i\in[1,t]$.
\item[(2)] Assume that $\mathcal{C}$ is Hom-finite.
Let $(\Theta, \mathbb{P}')$ be a mixed stratifying system in $\mathcal{C}$ and $\Psi:=\mathcal{C}(\mathbb{P}', -)$.
Then the endomorphism algebra $\End_{\mathcal{C}}(\mathbb{P}')$ is a mixed stratified algebra with respect to $(\Psi (\Theta(1)), \Psi(\Theta(2)), \ldots, \Psi(\Theta(t)))$ and $\mathcal{F}(\Theta)$ is equivalent to $\mathcal{F}(\Psi(\Theta))$.
\end{itemize}
\end{theorem}

Brundan and Stroppel \cite{BS} gave Ringel duality of mixed stratified algebras. 
This is a generalization of \cite{R91, AHLU00, X02}. 
Let $A$ be a mixed stratified algebra with respect to $\Xi$.
Then $\mathcal{F}(\Xi)$ admits an injective cogenerator $\mathbb{T}$, which is a Wakamatsu tilting module but not necessarily tilting, and the enodomorphism algebra $C:=\End_{A}(\mathbb{T})^{\op}$ is also a mixed stratified algebra. Moreover, $\Phi(A)$ is an injective cogenerator in $\mathcal{F}(\Phi(\Xi))$ and $\End_{C}(\Phi(A))^{\op}\cong A$, where $\Phi:=\Hom_{A}(-,\mathbb{T})$.
Thus $C$ is called a \emph{Ringel dual} of $A$.
Focusing on the property that $(\Xi, A,\mathbb{T})$ is a mixed bistratifying system (see Definition \ref{def_mss}), we study Ringel duality of mixed stratified algebras from the viewpoint of mixed bistratifying systems.
The following theorem gives a framework of Ringel duality.

\begin{theorem}[Theorem \ref{thm_rdual}] \label{intro_thm2}
Let $(\Theta, \mathbb{P}, \mathbb{I})$ be a mixed bistratifying system in $\mathcal{C}$.
Then the following statements hold.
\begin{itemize}
\item[(1)] Let $C:=\End_{\mathcal{C}}(\mathbb{I})^{\op}$ and $\Phi:=\mathcal{C}(-,\mathbb{I})$. Then $(\Phi(\Theta), \Phi(\mathbb{I}), \Phi(\mathbb{P}))$ is a mixed bistratifying system in $\mod C$.
Moreover, $\Phi(\mathbb{P})$ is a Wakamatsu titling $C$-module and $(\Phi'\Phi(\Theta), \Phi'\Phi(\mathbb{P}), \Phi'\Phi(\mathbb{I})) \cong (\Psi(\Theta), \Psi(\mathbb{P}), \Psi(\mathbb{I}))$ holds, where $\Phi':=\Hom_{C}(-,\Phi(\mathbb{P}))$.
In particular, $\End_{\mathcal{C}}(\mathbb{P}) \cong \End_{C}(\Phi(\mathbb{P}))$.
\item[(2)] Let $B:=\End_{\mathcal{C}}(\mathbb{P})$ and $\Psi:=\mathcal{C}(\mathbb{P},-)$. 
Then $(\Psi(\Theta), \Psi(\mathbb{P}), \Psi(\mathbb{I}))$ is a mixed bistratifying system in $\mod B$.
Moreover, $\Psi(\mathbb{I})$ is a Wakamatsu titling $B$-module and $(\Psi'\Psi(\Theta),\Psi'\Psi(\mathbb{I}), \Psi'\Psi(\mathbb{P}))\cong (\Phi(\Theta),\Phi(\mathbb{I}), \Phi(\mathbb{P}))$ holds, where $\Psi':=\Hom_{B}(-,\Psi(\mathbb{I}))$.
In particular, $\End_{\mathcal{C}}(\mathbb{I}) \cong \End_{B}(\Psi(\mathbb{I}))$.
\end{itemize}
\end{theorem}

\section{Preliminaries}
Throughout this paper, we assume that every category is skeletally small, that is, the isomorphism classes of objects form a set.
In addition, all subcategories are assumed to be full and closed under isomorphisms.

In this section, we collect terminologies, properties and examples of extriangulated categories.
We omit the precise definition of extriangulated categories.
For details, we refer to \cite{NP19, INP}.
Let $\mathcal{C}$ be an additive category.
We say that two complexes $A\xrightarrow{f}B\xrightarrow{g}C$ and $A\xrightarrow{f'}B'\xrightarrow{g'}C$ in $\mathcal{C}$ are \emph{equivalent} if there is an isomorphism $b:B\to B'$ such that the diagram
\begin{align}
\xymatrix{
A\ar[r]^-{f}\ar@{=}[d]&B\ar[r]^-{g}\ar[d]_{b}^{\cong}&C\ar@{=}[d]\\
A\ar[r]^-{f'}&B'\ar[r]^-{g'}&C
}\notag
\end{align}
is commutative, and let $[A\xrightarrow{f}B\xrightarrow{g}C]$ denote the equivalence class of $A\xrightarrow{f}B\xrightarrow{g}C$.

An extriangulated category $\mathcal{C}=(\mathcal{C},\mathbb{E},\mathfrak{s})$
consists of the following data which satisfy certain axioms (see \cite[Definition 2.12]{NP19}).
\begin{itemize}
\item $\mathcal{C}$ is an additive category.
\item $\mathbb{E}:\mathcal{C}^{\mathrm{op}}\times \mathcal{C}\to \mathcal{A}b$ is an additive bifunctor, where $\mathcal{A}b$ denotes the category of abelian groups.
\item $\mathfrak{s}$ is a correspondence which associates an equivalence class $[A\rightarrow B\rightarrow C]$ of complexes in $\mathcal{C}$ to each $\delta\in \mathbb{E}(C,A)$.
\end{itemize}
We call a complex $A\xrightarrow{f}B\xrightarrow{g}C$ in $\mathcal{C}$ an \emph{$\mathfrak{s}$-conflation} if there exists $\delta\in \mathbb{E}(C,A)$ such that $\mathfrak{s}(\delta)=[A\xrightarrow{f}B\xrightarrow{g}C]$, and often write the $\mathfrak{s}$-conflation as $A\xrightarrow{f}B\xrightarrow{g}C\overset{\delta}{\dashrightarrow}$.
Then $f$ is called an \emph{$\mathfrak{s}$-inflation}, $g$ is called an \emph{$\mathfrak{s}$-deflation}, $C$ is denoted by $\cone(f)$ and $A$ is denoted by $\cocone(g)$.

Let $\delta\in \mathbb{E}(C,A)$.
By Yoneda's lemma, we have two natural transformations $\delta_{\sharp}:\mathcal{C}(-,C)\to \mathbb{E}(-,A)$ and $\delta^{\sharp}:\mathcal{C}(A,-)\to\mathbb{E}(C,-)$, that is, for each $W\in\mathcal{C}$,
\begin{align}
&(\delta_{\sharp})_{W}: \mathcal{C}(W,C)\to \mathbb{E}(W,A)\ \ (\varphi\mapsto\mathbb{E}(\varphi,A)(\delta)),\notag\\
&\delta^{\sharp}_{W}: \mathcal{C}(A,W)\to \mathbb{E}(C,W)\ \ (\varphi\mapsto\mathbb{E}(C,\varphi)(\delta)).\notag
\end{align}
Any $\mathfrak{s}$-conflation induces two long exact sequences in $\mathcal{A}b$.

\begin{proposition}[{\cite[Corollary 3.12]{NP19}}]\label{prop_longex}
Let $\mathcal{C}$ be an extriangulated category and
$A\xrightarrow{f}B\xrightarrow{g}C\overset{\delta}{\dashrightarrow}$ an $\mathfrak{s}$-conflation.
Then, for each $W\in\mathcal{C}$, two sequences
\begin{align}
&\mathcal{C}(W,A)\xrightarrow{\mathcal{C}(W,f)}\mathcal{C}(W,B)\xrightarrow{\mathcal{C}(W,g)}\mathcal{C}(W,C)\xrightarrow{(\delta_{\sharp})_{W}}\mathbb{E}(W,A)\xrightarrow{\mathbb{E}(W,f)}\mathbb{E}(W,B)\xrightarrow{\mathbb{E}(W,g)}\mathbb{E}(W,C),\notag\\
&\mathcal{C}(C,W)\xrightarrow{\mathcal{C}(g,W)}\mathcal{C}(B,W)\xrightarrow{\mathcal{C}(f,W)}\mathcal{C}(A,W)\xrightarrow{\delta^{\sharp}_{W}}\mathbb{E}(C,W)\xrightarrow{\mathbb{E}(g,W)}\mathbb{E}(B,W)\xrightarrow{\mathbb{E}(f,W)}\mathbb{E}(A,W) \notag
\end{align}
are exact.
\end{proposition}

It follows from \cite[Remark 2.16]{NP19} that if $\left[\begin{smallmatrix}f& 0\end{smallmatrix}\right]:A_{1}\oplus A_{2} \rightarrow B$ is an $\mathfrak{s}$-inflation, then so is $f$.
We discuss the case where $\left[\begin{smallmatrix}f\\ 0\end{smallmatrix}\right]:A\rightarrow B_{1}\oplus B_{2}$ is an $\mathfrak{s}$-inflation.
Recall that an additive category $\mathcal{C}$ is said to be \emph{idempotent complete} if each idempotent morphism in $\mathcal{C}$ admits a kernel.
Note that Krull--Schmidt categories are idempotent complete.

\begin{lemma} \label{lem_minmor}
Let $\mathcal{C}$ be an idempotent complete extriangulated category.
Then the following statements hold.
\begin{itemize}
\item[(1)] If $\left[\begin{smallmatrix}f\\ 0\end{smallmatrix}\right]:A\rightarrow B_{1}\oplus B_{2}$ is an $\mathfrak{s}$-inflation, then so is $f$.
Moreover, we have an isomorphism of $\mathfrak{s}$-conflations
\begin{align}
\xymatrix{
A \ar[r]^-{\tiny\left[\begin{smallmatrix}f\\
0\end{smallmatrix}\right]}\ar@{=}[d]&B_{1} \oplus B_{2} \ar[r]\ar@{=}[d]&\cone\left( {\tiny\left[\begin{smallmatrix}f\\
0\end{smallmatrix}\right]}\right) \ar[d]^-{\cong}\ar@{-->}[r]&\;\\
A \ar[r]^-{\tiny\left[\begin{smallmatrix}f\\
0\end{smallmatrix}\right]}&B_{1}\oplus B_{2}\ar[r]^-{\tiny\left[\begin{smallmatrix}g_{1} & 0\\
0 &1\end{smallmatrix}\right]}&\cone (f)\oplus B_{2}\ar@{-->}[r]^-{\tiny\left[\begin{smallmatrix}\delta_{1} & 0\end{smallmatrix}\right]}&.
}\notag
\end{align}
\item[(2)] If $\left[\begin{smallmatrix}f_{1} & 0\\
0 &f_{2}\end{smallmatrix}\right]:A_{1} \oplus A_{2} \rightarrow B_{1} \oplus B_{2}$ is an $\mathfrak{s}$-inflation, then so are $f_{1}$ and $f_{2}$.
Moreover, we have an isomorphism of $\mathfrak{s}$-conflations
\begin{align}
\xymatrix{
A_{1} \oplus A_{2} \ar[r]^-{\tiny\left[\begin{smallmatrix}f_{1} & 0\\
0 &f_{2}\end{smallmatrix}\right]}\ar@{=}[d]&B_{1} \oplus B_{2} \ar[r]\ar@{=}[d]&\cone\left( {\tiny\left[\begin{smallmatrix}f_{1} & 0\\
0 &f_{2}\end{smallmatrix}\right]}\right) \ar[d]^-{\cong}\ar@{-->}[r]&\;\\
A_{1} \oplus A_{2}\ar[r]^-{\tiny\left[\begin{smallmatrix}f_{1} & 0\\
0 &f_{2}\end{smallmatrix}\right]}&B_{1}\oplus B_{2}\ar[r]^-{\tiny\left[\begin{smallmatrix}g_{1} & 0\\
0 &g_{2}\end{smallmatrix}\right]}&\cone (f_{1})\oplus \cone({f}_{2})\ar@{-->}[r]^-{\tiny\left[\begin{smallmatrix}\delta_{1} & 0\\
0 &\delta_{2}\end{smallmatrix}\right]}&.
}\notag
\end{align}
\end{itemize}
\end{lemma}

\begin{proof}
(1) By the assumption, there exists an $\mathfrak{s}$-conflation $A\xrightarrow{\tiny\left[\begin{smallmatrix}f\\ 0\end{smallmatrix}\right]} B_{1}\oplus B_{2} \xrightarrow{\tiny\left[\begin{smallmatrix}g'_{1} & g'_{2}\end{smallmatrix}\right]} C \overset{\delta}\dashrightarrow$.
Applying $\mathcal{C}(-, B_{2})$ to the $\mathfrak{s}$-conflation induces an exact sequence
\begin{align}
\mathcal{C}(C, B_{2}) \to \mathcal{C}(B_{1} \oplus B_{2}, B_{2}) \to \mathcal{C}(A, B_{2}). \notag
\end{align}
Thus we obtain $\varphi \in \mathcal{C}(C, B_{2})$ with $\varphi g'_{1}=0$ and $\varphi g'_{2}=1_{B_{2}}$.
Since $g'_{2} \varphi$ is an idempotent morphism in an idempotent complete category $\mathcal{C}$, the
kernel $C':=\ker (g'_{2}\varphi)$ exists.
Then there exist $p: C \to C'$ and $i: C'\to C$ such that $pg'_{2}=0$ and $\left[\begin{smallmatrix}i&g'_{2}\end{smallmatrix}\right]: C' \oplus B_{2} \to C$ is an isomorphism with inverse $\left[\begin{smallmatrix}p\\ \varphi \end{smallmatrix}\right]: C \to C' \oplus B_{2}$.
By \cite[Proposition 3.7]{NP19}, we obtain an $\mathfrak{s}$-conflation $A  \xrightarrow{\tiny\left[\begin{smallmatrix}f\\ 0  \end{smallmatrix}\right]} B_{1} \oplus B_{2} \xrightarrow{\tiny\left[\begin{smallmatrix}g_{1} & 0\\0 & 1 \end{smallmatrix}\right]} C' \oplus B_{2} \overset{\tiny\left[\begin{smallmatrix}\delta_{1} & \delta_{2} \end{smallmatrix}\right]}\dashrightarrow$, where $g_{1} :=p g'_{1}$, $\delta_{1}:=\mathbb{E}(i,A)(\delta) \in \mathbb{E}(C', A)$ and $\delta_{2}:=\mathbb{E}(g'_{2},A)(\delta) \in \mathbb{E}(B_{2}, A)$.
Due to \cite[Lemma 3.2]{NP19}, we have $\delta_{2}=0$.
It follows from the additivity of $\mathfrak{s}$ that $A \xrightarrow{f} B_{1} \xrightarrow{g_{1}} C' \overset{\delta_{1}} \dashrightarrow$ is an $\mathfrak{s}$-conflation.

(2) We only prove that $f_{1}$ is an $\mathfrak{s}$-inflation since the proof of $f_{2}$ is similar.
It follows from \cite[Remark 2.16]{NP19} that $\left[ \begin{smallmatrix}f_{1}\\ 0 \end{smallmatrix}\right] = \tiny\left[\begin{smallmatrix}f_{1}&0\\ 0&f_{2}\end{smallmatrix}\right] \left[\begin{smallmatrix}1\\ 0 \end{smallmatrix}\right]$ is an $\mathfrak{s}$-inflation, and hence so is $f_{1}$ by (1).
\end{proof}

By Lemma \ref{lem_minmor}, we consider decompositions of morphisms as follows.

\begin{remark}\label{rem_minmor}
Recall that any morphism $f: M\to N$ in a Krull--Schmidt extriangulated category has two decompositions $f=\left[\begin{smallmatrix}f'\\0 \end{smallmatrix}\right]: M\to N'\oplus N''$ with $f'$ left minimal and $f=\left[\begin{smallmatrix}0& f'' \end{smallmatrix}\right]: M'\oplus M''\to N$ with $f''$ right minimal.
If $f$ is an $\mathfrak{s}$-inflation (respectively, $\mathfrak{s}$-deflation), then $f'$ and $f''$ are also $\mathfrak{s}$-inflations (respectively, $\mathfrak{s}$-deflations).
In this case, by Lemma \ref{lem_minmor}, $\cone (f')$ and $\cone (f'')$ (respectively, $\cocone (f')$ and $\cocone (f'')$) are direct summands of $\cone(f)$ (respectively, $\cocone (f)$).
\end{remark}

We define extension-closed subcategories of extriangulated categories.

\begin{definition}
Let $\mathcal{C}$ be an extriangulated category.
\begin{itemize}
\item[(1)] For two subcategories $\mathcal{X}$ and $\mathcal{Y}$ of  $\mathcal{C}$, let $\mathcal{X}\ast\mathcal{Y}$ denote the subcategory of $\mathcal{C}$ consisting of $M\in\mathcal{C}$ which admits an $\mathfrak{s}$-conflation $X\rightarrow M\rightarrow Y\dashrightarrow$ with $X\in \mathcal{X}$ and $Y\in\mathcal{Y}$.
\item[(2)] A subcategory $\mathcal{C}'$ of $\mathcal{C}$ is said to be \emph{extension-closed} if $\mathcal{C}'\ast \mathcal{C}'\subseteq \mathcal{C}'$.
\end{itemize}
\end{definition}

In the following, we collect properties of the operation $\ast$.
Note that the operation $\ast$ is associative, that is, for three subcategories $\mathcal{X},\mathcal{Y},\mathcal{Z}$ of $\mathcal{C}$, the equality $(\mathcal{X}\ast \mathcal{Y})\ast \mathcal{Z}=\mathcal{X}\ast(\mathcal{Y}\ast \mathcal{Z})$ holds by (ET4) and (ET4)$^{\op}$ in \cite[Definition 2.12]{NP19}.
For $i \in \{1, 2 \}$, let $\mathcal{X}_{i}, \mathcal{Y}_{i}$ be subcategories of $\mathcal{C}$.
If $\mathcal{C}(\mathcal{X}_{i}, \mathcal{Y}_{j})=0$ for each $i, j \in \{1, 2\}$, then $\mathcal{C}(\mathcal{X}_{1} \ast \mathcal{X}_{2}, \mathcal{Y}_{1} \ast \mathcal{Y}_{2})=0$ by Proposition \ref{prop_longex}.
Similarly, if $\mathbb{E}(\mathcal{X}_{i}, \mathcal{Y}_{j})=0$ for each $i, j \in \{1, 2\}$, then $\mathbb{E}(\mathcal{X}_{1} \ast \mathcal{X}_{2}, \mathcal{Y}_{1} \ast \mathcal{Y}_{2})=0$.

\begin{lemma}\label{lem_extcl}
Let $\mathcal{X},\mathcal{Y}$ be subcategories of $\mathcal{C}$ which are closed under direct summands.
Then the following statements hold.
\begin{itemize}
\item[(1)] If $\mathbb{E}(\mathcal{X},\mathcal{Y})=0$, then $\mathcal{Y}\ast \mathcal{X}\subseteq \mathcal{X}\ast\mathcal{Y}$.
\item[(2)] Assume that $\mathcal{C}$ is a Krull--Schmidt category.
If $\mathcal{C}(\mathcal{X},\mathcal{Y})=0$, then $\mathcal{X}\ast \mathcal{Y}$ is closed under direct summands.
\end{itemize}
\end{lemma}

\begin{proof}
(1) Let $Z \in \mathcal{Y} \ast \mathcal{X}$.
Then there exists an $\mathfrak{s}$-conflation $Y \to Z \to X \dashrightarrow$ such that $Y \in \mathcal{Y}$ and $X \in \mathcal{X}$.
By $\mathbb{E}(\mathcal{X}, \mathcal{Y})=0$, the $\mathfrak{s}$-conflation splits, and hence $Z \cong X \oplus Y \in \mathcal{X} \ast \mathcal{Y}$.

(2) Let $Z:=Z_{1}\oplus Z_{2}\in \mathcal{X} \ast \mathcal{Y}$.
We show $Z_{1},Z_{2}\in \mathcal{X}\ast \mathcal{Y}$.
Take an $\mathfrak{s}$-conflation $X \overset{f}{\rightarrow} Z \overset{g}{\rightarrow} Y \dasharrow$ with $X\in \mathcal{X}$ and $Y\in \mathcal{Y}$.
Since $\mathcal{C}$ is a Krull--Schmidt category and $\mathcal{X}, \mathcal{Y}$ are closed under direct summands, we may assume that $f$ is a right minimal morphism by Remark \ref{rem_minmor}.
Then $f$ is a minimal right $\mathcal{X}$-approximation of $Z$ by $\mathcal{C}(\mathcal{X}, \mathcal{Y})=0$.
Let $p_{i}: Z \to Z_{i}$ be the projection for each $i\in\{1,2\}$.
Since $p_{i}f: X \to Z_{i}$ is a right $\mathcal{X}$-approximation, we can take a minimal right $\mathcal{X}$-approximation $f_{i}:X_{i}\rightarrow Z_{i}$ of $Z_{i}$.
Then $\tiny\left[\begin{smallmatrix}f_{1}&0\\ 0&f_{2}\end{smallmatrix}\right]:X_{1}\oplus X_{2}\rightarrow Z_{1}\oplus Z_{2}$ is also a minimal right $\mathcal{X}$-approximation of $Z$.
By the uniqueness of a minimal right $\mathcal{X}$-approximation, we have $X_{1} \oplus X_{2} \cong X$, and hence  $\tiny\left[\begin{smallmatrix}f_{1}&0\\ 0&f_{2}\end{smallmatrix}\right]$ is an $\mathfrak{s}$-inflation by \cite[Proposition 3.7]{NP19}.
Since it follows from Lemma \ref{lem_minmor}(2) that $f_{1},f_{2}$ are $\mathfrak{s}$-inflations and $\cone (f_{1}) \oplus \cone (f_{2}) \cong Y$ holds, we have an $\mathfrak{s}$-conflation $X_{i} \xrightarrow{f_{i}} Z_{i} \to \cone (f_{i}) \dashrightarrow$ with $X_{i} \in \mathcal{X}$ and $\cone(f_{i}) \in \mathcal{Y}$.
Thus $Z_{1}, Z_{2} \in \mathcal{X}\ast \mathcal{Y}$.
\end{proof}

For an object $M \in \mathcal{C}$, we define a subcategory $\mathcal{F}(M)$ as
\begin{align}
\mathcal{F}(M)=\bigcup_{l>0}\underbrace{\add M\ast \cdots \ast \add M}_{\textnormal{$l$ times}}, \notag
\end{align}
where $\add M$ denotes the subcategory of $\mathcal{C}$ whose objects are direct summands of finite direct sums of $M$.
Note that $\mathcal{F}(M)$ is an extension-closed subcategory of $\mathcal{C}$.
We give a description of $\mathcal{F}(M)$ as follows.

\begin{lemma}\label{lem_filt}
Let $\displaystyle M:=\oplus_{i=1}^{n}M_{i}$ be an object in $\mathcal{C}$, where each $M_{i}$ is indecomposable.
Then we have
\begin{align}
\displaystyle \mathcal{F}(M)=\bigcup_{\substack{i_{1},\ldots, i_{l} \in [1,n]\\l>0}} \add M_{i_{1}} \ast \cdots \ast \add M_{i_{l}}. \notag
\end{align}
\end{lemma}

\begin{proof}
By definition, $\displaystyle \mathcal{F}(M) \supset \bigcup_{\substack{i_{1},\ldots, i_{l} \in [1,n]\\l >0}} \add M_{i_{1}} \ast \cdots \ast \add M_{i_{l}}$ holds.
We show the converse inclusion.
Let $X \in \mathcal{F}(M)$.
Then there exists a sequence of $\mathfrak{s}$-inflations
\begin{align}
0=X_{0}\xrightarrow{f_{0}=0} X_{1}\xrightarrow{f_{1}} X_{2}\xrightarrow{f_{2}} \cdots \xrightarrow{f_{l-1}} X_{l}=X \notag
\end{align}
such that $C_{i}:=\cone (f_{i-1}) \in \add M$ for each $i\in [1,l]$.
Consider the $\mathfrak{s}$-conflation $X_{i-1} \xrightarrow{f_{i-1}} X_{i} \to C_{i} \dashrightarrow$.
We decompose $C_{i}$ as $C_{i}=M_{i_{1}}^{\oplus m_{1}} \oplus M_{i_{2}}^{\oplus m_{2}} \oplus \cdots \oplus M_{i_{k}}^{\oplus m_{k}}$, where all $m_{p} \neq 0$.
Let $C_{i}^{j}:=M_{i_{j+1}}^{\oplus m_{j+1}} \oplus \cdots \oplus M_{i_{k}}^{\oplus m_{k}}$.
Applying (ET4)$^{\op}$ in \cite[Definition 2.12]{NP19} to $X_{i-1} \xrightarrow{f_{i-1}} X_{i} \to C_{i} \dashrightarrow$ and $M^{\oplus m_{1}}_{i_{1}} \to C_{i} \to C_{i}^{1}\overset{0}\dashrightarrow$ induces a commutative diagram of $\mathfrak{s}$-conflations
\begin{align}
\xymatrix{
X_{i-1} \ar[r]^{g_{1}} \ar@{=}[d] &X_{i}^{1} \ar[r]\ar[d]&M_{i_{1}}^{\oplus m_{1}} \ar[d] \ar@{-->}[r] &\\
X_{i-1} \ar[r]^{f_{i-1}}&X_{i} \ar[r]\ar[d]&C_{i} \ar@{-->}[r]\ar[d]&.\\
&C_{i}^{1}\ar@{-->}[d]\ar@{=}[r] &C_{i}^{1}\ar@{-->}[d]^{0}&\\
&&&}\notag
\end{align}
Thus $X_{i}^{1} \to X_{i} \to C_{i}^{1} \dashrightarrow$ is an $\mathfrak{s}$-conflation.
Inductively, we obtain an $\mathfrak{s}$-conflation $X_{i}^{j-1} \to X_{i} \to C_{i}^{j-1} \dashrightarrow$.
Applying (ET4)$^{\op}$ in \cite[Definition 2.12]{NP19} to the $\mathfrak{s}$-conflation and $M_{i_{j}}^{\oplus m_{j}} \to C_{i}^{j-1} \to C_{i}^{j} \overset{0}\dashrightarrow$, we have a commutative diagram of $\mathfrak{s}$-conflations
\begin{align}
\xymatrix{
X_{i}^{j-1} \ar[r]^{g_{j}} \ar@{=}[d] &X_{i}^{j} \ar[r]\ar[d]&M_{i_{j}}^{\oplus m_{j}} \ar[d] \ar@{-->}[r] &\\
X_{i}^{j-1} \ar[r]&X_{i} \ar[r]\ar[d]&C_{i}^{j-1} \ar@{-->}[r] \ar[d]&.\\
&C_{i}^{j}\ar@{=}[r]\ar@{-->}[d] &C_{i}^{j}\ar@{-->}[d]^{0}&&\\
&&&
}\notag
\end{align}
Thus $X_{i}^{j-1} \xrightarrow{g_{j}} X_{i}^{j} \to M_{i_{j}}^{\oplus m_{j}} \dashrightarrow$ is an $\mathfrak{s}$-conflation.
Repeating this process, there exists a sequence of $\mathfrak{s}$-inflations
\begin{align}
X_{i-1}\xrightarrow{g_{1}} X_{i}^{1}\xrightarrow{g_{2}} X_{i}^{2}\xrightarrow{g_{3}} \cdots \xrightarrow{g_{k}} X^{k}_{i}=X_{i} \notag
\end{align}
such that $\cone (g_{j}) \cong M_{i_{j}}^{\oplus m_{j}}$ for each $j \in [1, k]$.
This completes the proof.
\end{proof}

The subcategory $\mathcal{F}(M)$ is not necessarily closed under direct summands (see \cite[page 210]{R91}).
However the following result is known.

\begin{proposition}[{\cite[Lemma 5.4(1)]{WWZ21}}] \label{prop_sbrick-dirsummand}
Let $\mathcal{C}$ be an idempotent complete extriangulated category.
If $S$ is a (semi)brick, then $\mathcal{F}(S)$ is closed under direct summands.
\end{proposition}

Now we recall a negative first extension structure on an extriangulated category \cite{AET}.

\begin{definition}
Let $\mathcal{C}$ be an extriangulated category.
A \emph{negative first extension structure} on $\mathcal{C}$ consists of the following data:
\begin{itemize}
\setlength{\itemindent}{20pt}
\item[(NE1)] $\mathbb{E}^{-1}:\mathcal{C}^{\mathrm{op}}\times \mathcal{C}\to \mathcal{A}b$ is an additive bifunctor.
\item[(NE2)] For each $\delta\in \mathbb{E}(C,A)$, there exist two natural transformations
\begin{align}
&\delta_{\sharp}^{-1}: \mathbb{E}^{-1}(-,C)\to \mathcal{C}(-,A),\notag\\
&\delta^{\sharp}_{-1}: \mathbb{E}^{-1}(A,-)\to \mathcal{C}(C,-)\notag
\end{align}
such that, for each $\mathfrak{s}$-conflation
$A\xrightarrow{f}B\xrightarrow{g}C\overset{\delta}{\dashrightarrow}$ and each $W\in \mathcal{C}$, two sequences
\begin{align}
&\mathbb{E}^{-1}(W,A)\xrightarrow{\mathbb{E}^{-1}(W,f)}\mathbb{E}^{-1}(W,B)\xrightarrow{\mathbb{E}^{-1}(W,g)}\mathbb{E}^{-1}(W,C)\xrightarrow{(\delta^{-1}_{\sharp})_{W}}\mathcal{C}(W,A)\xrightarrow{\mathcal{C}(W,f)}\mathcal{C}(W,B),\notag\\
&\mathbb{E}^{-1}(C,W)\xrightarrow{\mathbb{E}^{-1}(g,W)}\mathbb{E}^{-1}(B,W)\xrightarrow{\mathbb{E}^{-1}(f,W)}\mathbb{E}^{-1}(A,W)\xrightarrow{(\delta^{\sharp}_{-1})_{W}}\mathcal{C}(C,W)\xrightarrow{\mathcal{C}(g,W)}\mathcal{C}(B,W)\notag
\end{align}
are exact.
\end{itemize}
Then we call $\mathcal{C}=(\mathcal{C},\mathbb{E},\mathfrak{s},\mathbb{E}^{-1})$ an \emph{extriangulated category with a negative first extension}.
\end{definition}

Note that a negative first extension is a special case of partial $\delta$-functors in the sense of \cite[Definition 4.7]{GNP}.
Typical examples of extriangulated categories are triangulated categories and exact categories (see \cite[Example 2.13]{NP19}), and moreover both categories naturally admit negative first extension structures.
In the rest of this paper, unless otherwise stated, we always regard these categories as extriangulated categories with negative first extensions defined below.

\begin{example}\label{ex_neg}
\begin{itemize}
\item[(1)] A triangulated category $\mathcal{D}$ becomes an extriangulated category with a negative first extension by the following data.
\begin{itemize}
\item[$\bullet$] $\mathbb{E}(C,A):=\mathcal{D}(C,\Sigma A)$ for all $A,C\in \mathcal{D}$, where $\Sigma$ is a shift functor of $\mathcal{D}$.
\item[$\bullet$] For $\delta\in \mathbb{E}(C,A)$, we take a triangle $A\xrightarrow{f}B\xrightarrow{g}C\xrightarrow{\delta}\Sigma A$.
Then we define $\mathfrak{s}(\delta):=[A\xrightarrow{f}B\xrightarrow{g}C]$.
\item[$\bullet$] $\mathbb{E}^{-1}(C,A):=\mathcal{D}(C,\Sigma^{-1}A)$ for all $A,C\in \mathcal{D}$.
\item[$\bullet$] For an $\mathfrak{s}$-conflation $A\xrightarrow{f}B\xrightarrow{g}C\overset{\delta}\dashrightarrow$, we define two natural transformations $\delta_{\sharp}^{-1}$ and $\delta_{-1}^{\sharp}$ as follows: for $W\in \mathcal{D}$,
\begin{align}
&(\delta_{\sharp}^{-1})_{W}:\mathbb{E}^{-1}(W,C)=\mathcal{D}(W,\Sigma^{-1}C)\xrightarrow{\mathcal{D}(W,\Sigma^{-1}\delta)}\mathcal{D}(W,A),\notag\\
&(\delta_{-1}^{\sharp})_{W}:\mathbb{E}^{-1}(A,W)=\mathcal{D}(A,\Sigma^{-1}W)\cong \mathcal{D}(\Sigma A,W)\xrightarrow{\mathcal{D}(\delta,W)}\mathcal{D}(C,W).\notag
\end{align}
\end{itemize}
\item[(2)] An exact category $\mathcal{E}$ becomes an extriangulated category with a negative first extension by the following data.
\begin{itemize}
\item[$\bullet$] $\mathbb{E}(C,A)$ is the set of isomorphism classes of conflations in $\mathcal{E}$ of the form $0 \to A \to B \to C \to 0$ for $A,C \in \mathcal{E}$.
\item[$\bullet$] $\mathfrak{s}$ is the identity.
\item[$\bullet$] $\mathbb{E}^{-1}(C,A)=0$ for all $A,C\in \mathcal{E}$.
\item[$\bullet$] For each $W\in \mathcal{E}$, the maps $(\delta_{\sharp}^{-1})_{W}$ and $(\delta_{-1}^{\sharp})_{W}$ are zero.
\end{itemize}
\end{itemize}
\end{example}

\section{Mixed standardizable sets and mixed stratifying systems}
In this section, we study mixed standardizable sets and mixed stratifying systems in an extriangulated category with a negative first extension.
Let $\mathcal{C}$ be an additive category and $F:\mathcal{C}^{\mathrm{op}}\times \mathcal{C}\to \mathcal{A}b$ a bifunctor.
For a subcategory $\mathcal{X}$ of $\mathcal{C}$, we define a subcategory $\mathcal{X}^{\perp_{F}}$ as
\begin{align}
\mathcal{X}^{\perp_{F}}:=\{ M\in \mathcal{C} \mid F(\mathcal{X}, M)=0 \}.\notag
\end{align}
Dually, we define a subcategory ${}^{\perp_{F}}\mathcal{X}$.

\subsection{Universal extensions}
Let $\Bbbk$ be a field and $\mathcal{C}:=(\mathcal{C},\mathbb{E},\mathfrak{s})$ a $\Bbbk$-linear extriangulated category.
In this subsection, we study properties of universal extensions.
Let $D$ be a division ring and let $M,N$ be objects in $\mathcal{C}$.
Assume that $\mathbb{E}(M,N)\neq 0$ is a finite dimensional right $D$-module and let $d:=\dim_{D}\mathbb{E}(M,N)$.
Take a $D$-basis $\delta_{1},\delta_{2}, \ldots ,\delta_{d}$ of $\mathbb{E}(M,N)$.
Then we have an $\mathfrak{s}$-conflation $N\xrightarrow{f_{i}}E_{i}\xrightarrow{g_{i}}M\overset{\delta_{i}}{\dashrightarrow}$ for each $i \in [1, d]$.
It follows from \cite[Corollary 3.16]{NP19} that $\tiny \left[\begin{smallmatrix}1\\[-1.2mm]\vdots\\1\end{smallmatrix}\right]: M \to M^{\oplus d}$ is an $\mathfrak{s}$-inflation.
Thus, by (ET4)$^{\op}$ in \cite[Definition 2.12]{NP19}, we have a commutative diagram of $\mathfrak{s}$-conflations
\begin{align}
\xymatrix{
N^{\oplus d}\ar[r]^-{f}\ar@{=}[d]&E\ar[r]^-{g}\ar[d]&M\ar@{-->}[r]^{\delta}\ar[d]^-{\tiny\left[\begin{smallmatrix}1\\[-1.2mm]\vdots\\1\end{smallmatrix}\right]}&\;\\
N^{\oplus d}\ar[r]^-{\oplus f_{i}}&\oplus_{i=1}^{d}E_{i}\ar[r]^-{\oplus g_{i}}&M^{\oplus d}\ar@{-->}[r]^{\oplus \delta_{i}}&.
}\notag
\end{align}
Then we call the $\mathfrak{s}$-conflation $N^{\oplus d}\xrightarrow{f} E\xrightarrow{g} M\overset{\delta}\dashrightarrow$ a \emph{$D$-universal extension} (or simply, \emph{universal extension}) \emph{of $M$ by $N$}.
If $d=\dim_{D} \mathbb{E}(M, N)=0$, then we formally call the $\mathfrak{s}$-conflation $(N^{\oplus d}\xrightarrow{f} E\xrightarrow{g} M\overset{\delta}\dashrightarrow):=(0\rightarrow M\xrightarrow{1_{M}} M\overset{0}\dashrightarrow)$ a \emph{$D$-universal extension of $M$ by $N$}.
Note that $d\neq 0$ if and only if $\delta \neq 0$ if and only if $g$ is not an isomorphism.
If $D=\Bbbk$, then $\mathbb{E}(M,N)$ naturally becomes a $D$-module and $\delta_{N}^{\sharp}: \mathcal{C}(N^{\oplus d},N)\to \mathbb{E}(M,N)$ is an epimorphism.
If $D=\End_{\mathcal{C}}(N)^{\op}$ (i.e., $N$ is a brick), then $\mathbb{E}(M,N)$ naturally becomes a $D$-module and
the map $\delta_{N}^{\sharp}$ is an isomorphism.
Dually we define a \emph{$D$-universal coextension of $M$ by $N$}.

For an object $N \in \mathcal{C}$, we put
\begin{align}
D_{N}:=
\begin{cases}
\ \End_{\mathcal{C}}(N)^{\op} &(\textnormal{if $N$ is a brick}),\\
\ \Bbbk &(\textnormal{otherwise}).
\end{cases}\notag
\end{align}
Then $D_{N}$ is a division ring and $\mathbb{E}(M,N)$ is a right $D_{N}$-module for each $M\in\mathcal{C}$.
By $D_{N}$-universal extensions, we obtain infinitely many $\mathfrak{s}$-conflations.

\begin{lemma}\label{lem_brick-uniex}
Let $M,N$ be objects in $\mathcal{C}$ satisfying $\dim_{D_{N}}\mathbb{E}(M\oplus N, N)<\infty$.
Then the following statements hold.
\begin{itemize}
\item[(1)] There exists an infinite sequence of $\mathfrak{s}$-deflations
\begin{align}\label{seq-universal-extension}
\cdots \to P_{s} \xrightarrow{a_{s}} P_{s-1} \to \cdots \to P_{1} \xrightarrow{a_{1}} M
\end{align}
such that, for all $s\ge 1$,
\begin{itemize}
\item[(a)] $N^{\oplus d_{s}} \to P_{s} \xrightarrow{a_{s}}P_{s-1} \dashrightarrow$ is a $D_{N}$-universal extension of $P_{s-1}$ by $N$, where $d_{s}:=\dim_{D_{N}}\mathbb{E}(P_{s-1}, N)$ and $P_{0}:=M$,
\item[(b)] $\alpha_{s}:=a_{1}  \cdots  a_{s}$ is an $\mathfrak{s}$-deflation and $K_{s}:= \cocone(\alpha_{s}) \in \mathcal{F}(N)$,
\item[(c)] the map $(\eta_{s})_{N}^{\sharp}:\mathcal{C}(K_{s}, N) \to \mathbb{E}(M,N)$ is an epimorphism as $D_{N}$-modules, where $\mathfrak{s}(\eta_{s})=[ K_{s} \to P_{s} \xrightarrow{\alpha_{s}} M ]$.
If $N$ is a brick, then the map $(\eta_{s})_{N}^{\sharp}$ is an isomorphism as $D_{N}$-modules.
\end{itemize}
\item [(2)] Moreover, if there exists an integer $s\ge1$ such that $a_{s}$ is an isomorphism, then we have the  $\mathfrak{s}$-conflation
\begin{align}
K_{s}\to P_{s}\xrightarrow{\alpha_{s}} M\overset{\eta_{s}}\dashrightarrow \notag
\end{align}
such that $K_{s}\in\mathcal{F}(N)$, $P_{s}\in {}^{\perp_{\mathbb{E}}}\mathcal{F}(N)$ and $\alpha_{s}$ is a right ${}^{\perp_{\mathbb{E}}}\mathcal{F}(N)$-approximation of $M$.
In this case, $a_{t}$ is an isomorphism for all $t \ge s +1$.
\end{itemize}
\end{lemma}

\begin{proof}
For simplicity, let $D:=D_{N}$.

(1) We use induction on $s$.
Assume $s=1$. Let $d_{1}:= \dim_{D} \mathbb{E}(M,N)$.
Then a $D$-universal extension $N^{\oplus d_{1}} \to P_{1} \xrightarrow{a_{1}} M \dashrightarrow$ satisfies the conditions (a), (b) and (c).
Assume $s \ge 2$.
By the induction hypothesis, there exists a sequence of $\mathfrak{s}$-deflations
\begin{align}\label{pre-epi-seq}
P_{s-1} \xrightarrow{a_{s-1}} \cdots \to P_{1} \xrightarrow{a_{1}} M
\end{align}
which satisfies the conditions (a), (b) and (c).
Since $K_{s-1} \in \mathcal{F}(N)$ and $\dim_{D}\mathbb{E}(N, N) < \infty$, we obtain $\dim_{D}\mathbb{E}(K_{s-1}, N) < \infty$ by Proposition \ref{prop_longex}.
Applying $\mathbb{E}(-, N)$ to the $\mathfrak{s}$-conflation $K_{s-1} \to P_{s-1} \xrightarrow{\alpha_{s-1}} M \overset{\eta_{s-1}}{\dashrightarrow}$, we have $\dim_{D}\mathbb{E}(P_{s-1}, N) < \infty$.
Let $d_{s}:=\dim_{D} \mathbb{E}(P_{s-1}, N)$.
Then we have a $D$-universal extension  $N^{\oplus d_{s}} \to P_{s} \xrightarrow{a_{s}} P_{s-1} \overset{\delta_{s}}\dashrightarrow$.
Thus (a) holds.
Applying (ET4)$^{\op}$ in \cite[Definition 2.12]{NP19} to two $\mathfrak{s}$-conflations above induces a commutative diagram of $\mathfrak{s}$-conflations
\begin{align}
\xymatrix{
N^{\oplus d_{s}} \ar[r] \ar@{=}[d] &K_{s} \ar[r]^{\varphi_{s}}\ar[d]&K_{s-1} \ar[d] \ar@{-->}[r]^{\delta'_{s}}&\\
N^{\oplus d_{s}} \ar[r]&P_{s}\ar[r]^-{a_{s}}\ar[d]^-{\alpha_{s}}&P_{s-1} \ar@{-->}[r]^{\delta_{s}}\ar[d]^-{\alpha_{s-1}}&.\\
&M\ar@{=}[r]\ar@{-->}[d]^{\eta_{s}} &M \ar@{-->}[d]^{\eta_{s-1}}&&\\
&&&\notag
}
\end{align}
Since $\mathcal{F}(N)$ is extension-closed, we obtain $K_{s}=\cocone (\alpha_{s}) \in \mathcal{F}(N)$.
Hence (b) holds.
Applying $\mathcal{C}(-,N)$ to the commutative diagram above, we have a commutative diagram
\begin{align}
\xymatrix{
&&&\mathcal{C}(K_{s-1}, N)\ar[d]^-{(\eta_{s-1})_{N}^{\sharp}}&\\
&&&\mathbb{E}(M, N)\ar[d]&\\
\mathcal{C}(P_{s-1}, N)\ar[r]^-{\mathcal{C}(a_{s},N)}\ar[d]&\mathcal{C}(P_{s}, N)\ar[r]\ar[d]&\mathcal{C}(N^{\oplus d_{s}}, N)\ar[r]^-{(\delta_{s})^{\sharp}_{N}}\ar@{=}[d]&\mathbb{E}(P_{s-1}, N)\ar[d]&\\
\mathcal{C}(K_{s-1}, N)\ar[r]^-{\mathcal{C}(\varphi_{s},N)}\ar[d]^-{(\eta_{s-1})_{N}^{\sharp}}&\mathcal{C}(K_{s},N)\ar[r]\ar[d]^-{(\eta_{s})_{N}^{\sharp}}&\mathcal{C}(N^{\oplus d_{s}}, N)\ar[r]^-{(\delta'_{s})^{\sharp}_{N}}&\mathbb{E}(K_{s-1}, N). &\\
\mathbb{E}(M,N)\ar@{=}[r]&\mathbb{E}(M,N)&&&\notag
}
\end{align}
Thus we obtain $(\eta_{s})^{\sharp}_{N} \circ\mathcal{C}(\varphi_{s},N)=(\eta_{s-1})^{\sharp}_{N}$.
Since the induction hypothesis implies that  $(\eta_{s-1})^{\sharp}_{N}$ is an epimorphism, so is $(\eta_{s})^{\sharp}_{N}$.
Moreover, we show that if $N$ is a brick, then $(\eta_{s})_{N}^{\sharp}$ is an isomorphism.
By a property of $D$-universal extensions, $(\delta_{s})^{\sharp}_{N}$ is an isomorphism.
Since the induction hypothesis implies that $(\eta_{s-1})_{N}^{\sharp}$ is also an isomorphism, $(\delta'_{s})_{N}^{\sharp}$ and $\mathcal{C}(\varphi_{s}, N)$ are monomorphisms.
This implies that $\mathcal{C}(\varphi_{s}, N)$ is an isomorphism, and hence so is $(\eta_{s})_{N}^{\sharp}$.
Thus (c) holds.

(2) By a property of $D$-universal extensions, $a_{s}$ is an isomorphism if and only if $d_{s}=0$.
Since $d_{s}=\dim_{D}\mathbb{E}(P_{s-1},N)=0$, we have $P_{s} \cong P_{s-1} \in {}^{\perp_{\mathbb{E}}}(\add N) = {}^{\perp_{\mathbb{E}}} \mathcal{F}(N)$.
We show that $\alpha_{s}$ is a right ${}^{\perp_{\mathbb{E}}}\mathcal{F}(N)$-approximation of $M$.
Let $X \in {}^{\perp_{\mathbb{E}}} \mathcal{F}(N)$.
Applying $\mathcal{C}(X, -)$ to the $\mathfrak{s}$-conflation $K_{s}\to P_{s}\xrightarrow{\alpha_{s}} M\dashrightarrow$ induces an exact sequence
\begin{align}
\mathcal{C}(X, P_{s}) \xrightarrow{\mathcal{C}(X, \alpha_{s})} \mathcal{C} (X, M) \to \mathbb{E}(X, K_{s}).\notag
\end{align}
By $K_{s} \in \mathcal{F}(N)$, the right-hand side vanishes.
Hence we have the assertion.
\end{proof}

We give a remark for the condition in Lemma \ref{lem_brick-uniex}(2).
In general, there does not necessarily exist an integer $s$ such that $a_{s}$ in \eqref{seq-universal-extension} is an isomorphism.
Indeed, we assume that $\Bbbk$ is an algebraically closed field and
let $A=\Bbbk(\xymatrix{1\ar@<0.5ex>[r]\ar@<-0.5ex>[r]&2})$ be the path algebra of the Kronecker quiver.
For $A$-modules $M=N=(\xymatrix{\Bbbk\ar@<0.5ex>[r]^{1_{\Bbbk}}\ar@<-0.5ex>[r]_{0}&\Bbbk})$, there exists no integer $s$ such that $a_{s}$ in \eqref{seq-universal-extension} is an isomorphism.

Now we introduce the following notion.

\begin{definition}
Let $N$ be an indecomposable object in $\mathcal{C}$.
Let $\mathcal{E}$ be an extension-closed subcategory of $\mathcal{C}$ which contains $N$ and satisfies $\dim_{D_{N}}\mathbb{E}(M,N) < \infty$ for each $M \in \mathcal{E}$.
We say that $\mathcal{E}$ \emph{admits a finite sequence of universal extensions by $N$} if for each $M\in \mathcal{E}$, there exists an integer $s\ge1$ such that $a_{s}$ in \eqref{seq-universal-extension} is an isomorphism.
Dually, we define the notion of a \emph{finite sequence of universal coextensions}.
\end{definition}

If $\mathcal{E}$ admits a finite sequence of universal extensions by $N$, then for each $M\in\mathcal{E}$, there exists an $\mathfrak{s}$-conflation $K\to P\xrightarrow{\alpha} M\dashrightarrow$ with $K\in \mathcal{F}(N)$, $P\in {}^{\perp_{\mathbb{E}}}\mathcal{F}(N)\cap \mathcal{E}$ and $\alpha$ is a right ${}^{\perp_{\mathbb{E}}}\mathcal{F}(N)$-approximation of $M$.

We give sufficient conditions that $\mathcal{E}$ admits a finite sequence of universal extensions by an indecomposable object $N$. 
An object $N$ is called a stone if it is indecomposable and $\mathbb{E}(N,N)=0$. 

\begin{proposition}\label{prop_stone-univex}
Let $N$ be a stone in $\mathcal{C}$.
Let $\mathcal{E}$ be an extension-closed subcategory of $\mathcal{C}$ which contains $N$ and satisfies $\dim_{D_{N}}\mathbb{E}(M,N) < \infty$ for each $M \in \mathcal{E}$.
Then $\mathcal{E}$ admits a finite sequence of universal extensions by $N$.
\end{proposition}

\begin{proof}
Let $M \in \mathcal{E}$ and $d_{1}:=\dim_{D_{N}}\mathbb{E}(M, N)<\infty$.
Applying $\mathcal{C}(-,N)$ to a $D_{N}$-universal extension $N^{\oplus d_{1}}\to P_{1}\xrightarrow{a_{1}} M \overset{\delta_{1}}\dashrightarrow$ induces an exact sequence
\begin{align}
\mathcal{C}(N^{\oplus d_{1}},N) \xrightarrow{(\delta_{1}^{\sharp})_{N}}\mathbb{E}(M,N)\to\mathbb{E}(P_{1},N)\to\mathbb{E}(N^{\oplus d_{1}},N)=0,  \notag
\end{align}
where the last equality follows from the assumption that $N$ is a stone.
Since $(\delta_{1}^{\sharp})_{N}$ is an epimorphism, we have $\mathbb{E}(P_{1},N)=0$.
By $d_{2}:=\dim_{D_{N}}\mathbb{E}(P_{1},N)=0$, the morphism $a_{2}$ in \eqref{seq-universal-extension} is an isomorphism.
Thus $\mathcal{E}$ admits a finite sequence of universal extensions by $N$.
\end{proof}

Let $A$ be a finite dimensional $\Bbbk$-algebra and let $\mod A$ denote the category of finitely generated right $A$-modules.
For module categories, we have the following result.

\begin{proposition}\label{prop_brick-univex}
Let $N$ be a brick in $\mod A$ and $\mathcal{E}$ an extension-closed subcategory of $\mod A$ containing $N$.
Then $\mathcal{E}$ admits a finite sequence of universal extensions by $N$ if one of the following two conditions is satisfied.
\begin{itemize}
\item[(i)] There exists an integer $l>0$ such that for each indecomposable $A$-module in $\mathcal{E}$, its length is bounded by $l$.
\item[(ii)] $N$ satisfies $\Ext_{A}^{1}(N,\rad N)=0$.
\end{itemize}
\end{proposition}

\begin{proof}
Regarding $\mathcal{C}=\mod A$ as the extriangulated category by Example \ref{ex_neg}(2), we use the notation in the proof of Lemma \ref{lem_brick-uniex}(1).
For both conditions (i) and (ii), we show that there exists $s \geq 1$ such that $a_{s}$ is an isomorphism.
Suppose to the contrary that $a_{k}$ is not an isomorphism for each $k\ge 1$.
Since $d_{k+1}=\dim_{D_{N}}\Ext_{A}^{1}(P_{k},N) \neq 0$, we have an exact sequence $0 \to N^{\oplus d_{k+1}} \to K_{k+1} \xrightarrow{\varphi_{k+1}} K_{k} \to 0$.
Hence $\ell(K_{k})<\ell(K_{k+1})$ holds, where $\ell(X)$ is the length of a module $X$.
Thus we obtain
\begin{align}
0< \ell(K_{1})< \ell(K_{2})<\cdots<\ell(K_{t})< \cdots. \notag
\end{align}
In the following, for the conditions (i) and (ii), we show that there exists $\ell$ such that $\ell(K_{t})\le\ell$ for each $t\ge 1$.
Then this induces a contradiction.

(i) We decompose $K_{t}$ as $K_{t}=\oplus_{j=1}^{n_{t}}K_{t}^{j}$, where each $K_{t}^{j}$ is indecomposable.
By Proposition \ref{prop_sbrick-dirsummand}, we have $K_{t}^{j} \in \mathcal{F}(N)\subseteq\mathcal{E}$, and hence $\Hom_{A}(K_{t}^{j},N)\neq 0$ for each $j\in [1,n_{t}]$.
Then
\begin{align}
d_{1}&=\dim_{D_{N}}\Ext_{A}^{1}(M,N)\underset{\textnormal{Lemma\;} \ref{lem_brick-uniex}\textnormal{(c)}}=\dim_{D_{N}}\Hom_{A}(K_{t}, N)\notag\\
&=\sum_{j=1}^{n_{t}}\dim_{D_{N}}\Hom_{A}(K_{t}^{j},N)\ge n_{t} \neq 0.\notag
\end{align}
On the other hand, since $K_{t}^{j}$ is indecomposable, the inequality $\ell(K_{t}^{j})\le l$ holds by the assumption (i).
Therefore we obtain
\begin{align}
\ell(K_{t})=\sum_{j=1}^{n_{t}}\ell(K_{t}^{j})\le l n_{t}\le l d_{1}=:\ell. \notag
\end{align}

(ii) We show $\top K_{t}\cong \top K_{t+1}$ for each $t \ge1$.
By Yoneda's lemma, it is enough to show that $\Hom_{A}(\top \varphi_{t+1}, A/\rad A): \Hom_{A}(\top K_{t}, A/\rad A) \to \Hom_{A}(\top K_{t+1}, A/\rad A)$ is an isomorphism.
Assume $\Ext_{A}^{1}(N,\rad N)=0$.
By $K_{t}, K_{t+1} \in \mathcal{F}(N)$, we have $\Ext_{A}^{1}(K_{t},\rad N)=0=\Ext_{A}^{1}(K_{t+1},\rad N)$. Applying $\Hom(\varphi_{t+1},-)$ to an exact sequence $0 \to \rad N \to N \to \top N \to 0$ induces a commutative diagram
\begin{align}
\xymatrix{
0\ar[d]&0\ar[d]&\\
\Hom_{A}(K_{t}, N)\ar[r]\ar[d]^-{\Hom(\varphi_{t+1},N)}&\Hom_{A}(K_{t}, \top N)\ar[r]\ar[d]^-{\Hom(\varphi_{t+1},\top N)}&\Ext_{A}^{1}(K_{t}, \rad N)=0\\
\Hom_{A}(K_{t+1}, N)\ar[r]&\Hom_{A}(K_{t+1}, \top N)\ar[r]&\Ext_{A}^{1}(K_{t+1}, \rad N)=0.
}\notag
\end{align}
Since $\Hom(\varphi_{t+1},N)$ is an isomorphism by the proof of Lemma \ref{lem_brick-uniex}(1), so is $\Hom(\varphi_{t+1},\top N)$.
For each simple $A$-module $S\notin \add(\top N)$, we have $\Hom_{A}(K_{t},S)=0=\Hom_{A}(K_{t+1},S)$.
Thus $\Hom_{A}(\varphi_{t+1}, A/\rad A): \Hom_{A}(K_{t}, A/\rad A) \to \Hom_{A}(K_{t+1}, A/\rad A)$ is an isomorphism.
By applying $\Hom_{A}(-, A/\rad A)$ to the exact sequence $0\to \rad K_{i} \xrightarrow{\iota_{i}} K_{i} \xrightarrow{\pi_{i}} \top K_{i} \to 0$ for $i=t, t+1$, there exists an exact sequence
\begin{align}
0 \to \Hom_{A}(\top K_{i}, A/\rad A) \to \Hom_{A}(K_{i}, A/\rad A) \xrightarrow{\Hom(\iota_{i},A/\rad A)} \Hom_{A}(\rad K_{i}, A/\rad A). \notag
\end{align}
Since $\Hom(\iota_{i},A/\rad A)=0$ holds, the map $\Hom(\pi_{i},A/\rad A)$ is an isomorphism.
Thus the commutative relation
\begin{align}
\Hom(\pi_{t+1},A/\rad A)\circ\Hom(\top \varphi_{t+1},A/\rad A)=\Hom(\varphi_{t+1},A/\rad A)\circ\Hom(\pi_{t},A/\rad A) \notag
\end{align}
implies that $\Hom(\top\varphi_{t+1}, A/\rad A)$ is also an isomorphism.
Hence $\top\varphi_{t+1}:\top K_{t}\to \top K_{t+1}$ is an isomorphism for each $t\ge1$.
Let $P$ be a projective cover of $K_{1}$.
Since $P$ is also a projective cover of $K_{t}$, we have $\ell(K_{t}) \le \ell(P)=:\ell$.
The proof is complete.
\end{proof}

As will be explained in Proposition \ref{prop_mixstd}, each proper standard module satisfies the condition in Proposition \ref{prop_brick-univex}(ii).
The following examples show that the conditions (i) and (ii) in Proposition \ref{prop_brick-univex} are independent.

\begin{example}
Assume that $\Bbbk$ is an algebraically closed field.
\begin{itemize}
\item[(1)] Let $A$ be the $\Bbbk$-algebra defined by the quiver
\begin{align}
\xymatrix{1\ar@(lu,ld)_-{a}&2\ar[l]_-{b}}\notag
\end{align}
with a relation $a^{2}$. Then $\mod A$ has exactly 7 indecomposable modules (up to isomorphisms), and hence it satisfies the condition (i) in Proposition \ref{prop_brick-univex}.
On the other hand, $N:=e_{2}A/\soc e_{2}A$ does not satisfy $\Ext_{A}^{1}(N,\rad N)= 0$ since there exists a non-split exact sequence $0\to \rad N \to e_{2}A\to N\to  0$.
\item[(2)] Let $A$ be the radical square zero $\Bbbk$-algebra defined by the quiver
\begin{align}
\xymatrix{1\ar@(lu,ld)\ar@(ru,rd)}\notag
\end{align}
and $N:=\top e_{1}A$.
Then clearly $\Ext_{A}^{1}(N,\rad N)=0$ but $\Ext_{A}^{1}(N,N)\neq 0$.
On the other hand, by \cite[Theorem X.2.6]{ARS95}, $\mathcal{F}(N)=\mod A$ is representation-infinite.
The first Brauer--Thrall conjecture implies that $\mathcal{F}(N)$ does not satisfy the condition (i) in Proposition \ref{prop_brick-univex}.
\end{itemize}
\end{example}

When $N$ is a brick and stone, we compare two universal extensions: one is a $\Bbbk$-universal extension, and the other is a $D_{N}$-universal extension.
Note $\dim_{D_{N}}\mathbb{E}(M,N)\leq \dim_{\Bbbk}\mathbb{E}(M,N)$.

\begin{remark}
Assume that $\mathcal{C}$ is a Krull--Schmidt category.
Let $N$ be a brick and stone.
Let $M\in \mathcal{C}$ be an object satisfying $\dim_{\Bbbk}\mathbb{E}(M,N)<\infty$.
By taking a $\Bbbk$-universal extension and Remark \ref{rem_minmor}, we have an $\mathfrak{s}$-conflation
\begin{align}
N^{\oplus d}\to P\xrightarrow{\alpha}M\dashrightarrow, \notag
\end{align}
where $d\le \dim_{\Bbbk}\mathbb{E}(M,N)$ and $\alpha$ is right minimal.
On the other hand, by taking a $D_{N}$-universal extension and Remark \ref{rem_minmor}, we have an $\mathfrak{s}$-conflation
\begin{align}
N^{\oplus d'}\to P'\xrightarrow{\alpha'}M\dashrightarrow, \notag
\end{align}
where $d'\le \dim_{D_{N}}\mathbb{E}(M,N)$ and $\alpha'$ is right minimal.
Since $N$ is a stone, we have $P,P'\in {}^{\perp_{\mathbb{E}}}\mathcal{F}(N)$, and $\alpha,\alpha'$ are minimal right ${}^{\perp_{\mathbb{E}}}\mathcal{F}(N)$-approximations. This implies that $P\cong P'$ and $d=d'$.
\end{remark}

\subsection{Mixed standardizable sets}
Let $\Bbbk$ be a field and  $\mathcal{C}:=(\mathcal{C},\mathbb{E},\mathfrak{s},\mathbb{E}^{-1})$ a $\Bbbk$-linear Krull--Schmidt extriangulated category with a negative first extension.
In this subsection, we introduce the notion of a mixed standardizable set, which is a generalization of standardizable sets \cite{DR92} and proper pre-costratifying systems \cite{MPV14} in an abelian category, $\Theta$-systems \cite{MS16} in a triangulated category and $F$-systems \cite{S19} in an exact category.
In the rest of this paper, when we say that $\Theta:=(\Theta(i_{1}), \Theta(i_{2}), \ldots, \Theta(i_{t}))$ is an ordered set of objects, this means that $i_{1}<i_{2} < \cdots <i_{t}$. 
We frequently identify an ordered set $\Theta=(\Theta(i_{1}),\Theta(i_{2}), \ldots, \Theta(i_{t}))$ of objects with an object $\Theta=\Theta(i_{1})\oplus \Theta(i_{2})\oplus \cdots \oplus \Theta(i_{t})$.

\begin{definition}\label{def_ms}
Let $\Theta:=(\Theta(1), \Theta(2), \ldots, \Theta(t))$ be an ordered set of objects in $\mathcal{C}$.
\begin{itemize}
\item[(1)] $\Theta$ is called a \emph{mixed standardizable set} in $\mathcal{C}$ if it satisfies the following conditions.
\begin{itemize}
\item[(MS1)] For each $i\in [1,t]$, $\Theta(i)$ is a brick or a stone.
\item[(MS2)] If $i>j$, then $\mathcal{C}(\Theta(i), \Theta(j))=0$ holds.
\item[(MS3)] If $i>j$, then $\mathbb{E}(\Theta(i), \Theta(j))=0$ holds.
\item[(MS4)] $\mathbb{E}^{-1}(\Theta,\Theta)=0$.
\end{itemize}
\item[(2)] Assume that $\Theta(i)$ is indecomposable for each $i\in [1,t]$.
Let $\mathcal{E}$ be an extension-closed subcategory of $\mathcal{C}$ which contains $\Theta$ and satisfies that $\mathbb{E}(M,\Theta(i))$ is a finite dimensional right $D_{\Theta(i)}$-module for each $M \in \mathcal{E}$ and $i\in [1,t]$.
We say that \emph{$\mathcal{E}$ admits a finite sequence of universal extensions by $\Theta$} if, for each $i\in [1,t]$, the category $\mathcal{E}$ admits a finite sequence of universal extensions by $\Theta(i)$.
Dually, we define the notion of a \emph{finite sequence of universal coextensions by $\Theta$}.
\end{itemize}
\end{definition}

It follows from Proposition \ref{prop_stone-univex} and its dual statement that if $\Theta(i)$ is a stone for all $i\in[1,t]$ and $\dim_{\Bbbk}\mathbb{E}(\Theta, \Theta)< \infty$, then $\mathcal{F}(\Theta)$ admits a finite sequence of universal extensions by $\Theta$ and a finite sequence of universal coextensions by $\Theta$.

We give a typical example that $\mod A$ admits a finite sequence of universal extensions.
Recall the definitions of standard modules and proper standard modules.
Let $A$ be a finite dimensional $\Bbbk$-algebra and $(e_{1},e_{2},\ldots, e_{n})$ a complete ordered set of primitive orthogonal idempotents of $A$.
Let $\{S(i):=\top e_{i}A \mid 1 \le i \le n \}$ denote the set of  representatives of isomorphism classes of simple $A$-modules.
The $i$-th standard module $\Delta(i)$ is defined as the maximal factor module of $e_{i}A$ satisfying $\Delta(i) \in \mathcal{F}(S(1)\oplus \cdots \oplus S(i))$.
The $i$-th proper standard module $\overline{\Delta}(i)$ is defined as the maximal factor module of $\Delta(i)$ satisfying $\rad \overline{\Delta}(i) \in \mathcal{F}(S(1)\oplus \cdots \oplus S(i-1))$.
Then we can check that the standard module $\Delta(i)$ is a stone and the proper standard module $\overline{\Delta}(i)$ is a brick.

\begin{proposition}\label{prop_mixstd}
Let $A$ be a finite dimensional $\Bbbk$-algebra and $\Theta(i)\in \{ \Delta(i), \overline{\Delta}(i)\}$ for each $i \in [1,n]$.
Then $\Theta:=(\Theta(1), \ldots, \Theta(n))$ is a mixed standardizable set and $\mod A$ admits a finite sequence of universal extensions by $\Theta$.
\end{proposition}

\begin{proof}
We regard $\mod A$ as the extriangulated category with the negative first extension $\mathbb{E}^{-1}=0$ by Example \ref{ex_neg}(2).
First we show that $\Theta$ is a mixed standardizable set.
Since $\mathbb{E}^{-1}=0$, (MS4) clearly holds.
By definitions of standard modules and proper standard modules, (MS1) and (MS2) clearly hold.
We show (MS3).
Let $i>j$.
Applying $\Hom_{A}(-, \Theta(j))$ to an exact sequence $0\to K \to e_{i}A \to \Theta(i) \to 0$ induces an exact sequence
\begin{align*}
\Hom_{A}(K, \Theta(j)) \to \Ext_{A}^{1}(\Theta(i), \Theta(j)) \to \Ext_{A}^{1}(e_{i}A, \Theta(j))=0.
\end{align*}
By $\top K \in  \add \oplus_{k \ge i} S(k)$, we have $\Hom_{A}(K, \Theta(j))=0$, and hence $\Ext_{A}^{1}(\Theta(i), \Theta(j))=0$.

Next we show that $\mod A$ admits a finite sequence of universal extensions by $\Theta$.
Since standard modules are stones, it follows from Proposition \ref{prop_stone-univex} that $\mod A$ admits a finite sequence of universal extensions by each standard module.
Thus we prove that $\mod A$ admits a finite sequence of universal extensions by each proper standard module.
By Proposition \ref{prop_brick-univex}(ii), it is enough to show  $\Ext_{A}^{1}(\overline{\Delta}(i), \rad \overline{\Delta}(i))=0$ for each $i\in [1,n]$.
Applying $\Hom_{A}(-,\rad \overline{\Delta}(i))$ to an exact sequence $0 \to K \to e_{i}A \to \overline{\Delta}(i) \to 0$ induces an exact sequence
\begin{align}
\Hom_{A}(K, \rad \overline{\Delta}(i)) \to \Ext_{A}^{1}(\overline{\Delta}(i), \rad \overline{\Delta}(i)) \to\Ext_{A}^{1}(e_{i}A, \rad \overline{\Delta}(i))= 0. \notag
\end{align}
By $\top K \in  \add \oplus_{k \ge i} S(k)$ and $\rad \overline{\Delta}(i) \in \mathcal{F}(S(1)\oplus \cdots \oplus S(i-1))$, we have $\Hom_{A}(K, \rad \overline{\Delta}(i))=0$.
This completes the proof.
\end{proof}

For a mixed standardizable set $\Theta$, the subcategory $\mathcal{F}(\Theta)$ has the following property.

\begin{proposition}\label{prop_dirsummand}
Let $\Theta:=(\Theta(1), \Theta(2), \ldots, \Theta(t))$ be an ordered set of indecomposable objects in $\mathcal{C}$.
Then the following statements hold.
\begin{itemize}
\item[(1)] If \textnormal{(MS3)} is satisfied, then $\mathcal{F}(\Theta)=\mathcal{F}(\Theta(t)) \ast \mathcal{F}(\Theta(t-1)) \ast \cdots \ast \mathcal{F}(\Theta(1))$ holds.
\item[(2)] If $\Theta$ satisfies \textnormal{(MS1)}, \textnormal{(MS2)} and \textnormal{(MS3)}, then $\mathcal{F}(\Theta)$ is closed under direct summands.
\end{itemize}
\end{proposition}

\begin{proof}
(1) Since $\mathcal{F}(\Theta)$ is extension-closed, $\mathcal{F}(\Theta)\supseteq \mathcal{F}(\Theta(t)) \ast \cdots \ast \mathcal{F}(\Theta(1))$ holds.
We show the converse inclusion.
Let $M\in \mathcal{F}(\Theta)$.
By Lemma \ref{lem_filt}, we obtain
\begin{align}
M\in \add \Theta(i_{1})\ast \cdots \ast \add \Theta(i_{l}). \notag
\end{align}
By (MS3) and Lemma \ref{lem_extcl}(1), we can take $t\ge i_{1}\ge i_{2}\ge \cdots \ge i_{l}\ge 1$.
Hence we have the assertion.

(2) Due to (1), we have $\mathcal{F}(\Theta)=\mathcal{F}(\Theta(t)) \ast \cdots \ast \mathcal{F}(\Theta(1))$.
By (MS2) and Lemma \ref{lem_extcl}(2), it is enough to show that  $\mathcal{F}(\Theta(i))$ is closed under direct summands for each $i\in [1,t]$.
If $\Theta(i)$ is a stone, then $\mathcal{F}(\Theta(i))=\add \Theta(i)$.
On the other hand, if $\Theta(i)$ is a brick, then $\mathcal{F}(\Theta(i))$ is closed under direct summands by Proposition \ref{prop_sbrick-dirsummand}.
This completes the proof.
\end{proof}

A subcategory $\mathcal{E}$ of $\mathcal{C}$ is said to be \emph{$\mathbb{E}$-finite} if for each $M,N\in\mathcal{E}$, the $\Bbbk$-vector space $\mathbb{E}(M,N)$ is finite dimensional.
Remark that $\dim_{D_{N}}\mathbb{E}(M,N)\le \dim_{\Bbbk}\mathbb{E}(M,N)$ holds.
For an object $\Theta$ of $\mathcal{C}$, we can check that $\add\Theta$ is $\mathbb{E}$-finite if and only if $\mathcal{F}(\Theta)$ is $\mathbb{E}$-finite.

In the following, we fix an ordered set $\Theta=(\Theta(1), \ldots, \Theta(t))$ of indecomposable objects in $\mathcal{C}$.
For each $j\in [1,t]$, we define a subcategory $\mathcal{F}(\Theta(\ge j))$ as $\mathcal{F}(\Theta(\ge j)):=\mathcal{F}(\Theta(j) \oplus \cdots \oplus \Theta(t))$.
Similarly, we define subcategories $\mathcal{F}(\Theta(>j))$, $\mathcal{F}(\Theta(\le j))$, and $\mathcal{F}(\Theta(<j))$.
If $\mathcal{E}$ admits a finite sequence of universal extensions by $\Theta$, then we obtain a special $\mathfrak{s}$-conflation as follows.

\begin{proposition}\label{prop_extproj}
Assume that $\Theta$ satisfies \textnormal{(MS1)}, \textnormal{(MS2)} and \textnormal{(MS3)}.
Let $\mathcal{E}$ be an $\mathbb{E}$-finite extension-closed subcategory of $\mathcal{C}$ containing $\Theta$ and admitting a finite sequence of universal extensions by $\Theta$.
Then the following statements hold for each $M \in \mathcal{E}$.
\begin{itemize}
\item[(1)] There exists an $\mathfrak{s}$-conflation $K \to P \xrightarrow{\alpha_{M}} M \dashrightarrow$ such that $P \in {}^{\perp_{\mathbb{E}}}\mathcal{F}(\Theta)\cap \mathcal{E}$ and $\alpha_{M}$ is right minimal.
Moreover, $K=0$ if $M\in{}^{\perp_{\mathbb{E}}}\mathcal{F}(\Theta)$ or $K \in \mathcal{F}(\Theta(\ge j))$ if $M \not\in{}^{\perp_\mathbb{E}}\mathcal{F}(\Theta)$, where $j$ is the minimum value satisfying $\mathbb{E}(M,\Theta(j))\neq 0$.
\item[(2)] Assume that $\Theta$ is a mixed standardizable set  and $\mathcal{E}=\mathcal{F}(\Theta)$.
If $M\not\in{}^{\perp_\mathbb{E}}\mathcal{F}(\Theta)$ is an indecomposable object contained in $\mathcal{F}(\Theta(\ge j))^{\perp_{\mathcal{C}}}$ or $\add \Theta(j)$, then $P$ is indecomposable.
\end{itemize}
\end{proposition}

To show Proposition \ref{prop_extproj}(2), we need the following lemma.

\begin{lemma}\label{lem_indec}
Let $\alpha: X \to Y$ be a right minimal morphism such that $\mathcal{C}(\alpha, Y)$ is an isomorphism.
If $\End_{\mathcal{C}}(Y)$ is local, then $X$ is indecomposable.
\end{lemma}

\begin{proof}
Suppose that $X$ is not indecomposable, that is, $X=X_{1}\oplus X_{2}$.
We decompose $\alpha$ as $\alpha=\left[\begin{smallmatrix}\alpha_{1}&\alpha_{2}\end{smallmatrix}\right]: X_{1}\oplus X_{2}\to Y$.
By our assumption, there uniquely exist morphisms $a_{1},a_{2}\in \mathcal{C}(Y,Y)$ such that $\left[\begin{smallmatrix}a_{1}\alpha_{1} & a_{1}\alpha_{2}\end{smallmatrix}\right]=\left[\begin{smallmatrix}\alpha_{1}&0\end{smallmatrix}\right]$ and $\left[\begin{smallmatrix}a_{2}\alpha_{1}&a_{2}\alpha_{2}\end{smallmatrix}\right]=\left[\begin{smallmatrix}0&\alpha_{2}\end{smallmatrix}\right]$.
Then $1_{Y}=a_{1}+a_{2}$.
Moreover, $a_{1}$ and $a_{2}$ are idempotents by
\begin{align}
a_{1}^{2}\mapsto \begin{bmatrix} a_{1}^{2}\alpha_{1} & a_{1}^{2}\alpha_{2}\end{bmatrix}=\begin{bmatrix}\alpha_{1}& 0\end{bmatrix},\
a_{2}^{2}\mapsto \begin{bmatrix} a_{2}^{2}\alpha_{1}& a_{2}^{2}\alpha_{2}\end{bmatrix}=\begin{bmatrix}0 & \alpha_{2}\end{bmatrix}.\notag
\end{align}
Since $\End_{\mathcal{C}}(Y)$ is local, we obtain either $\alpha_{1}=0$ or $\alpha_{2}=0$.
This contradicts the right minimality of $\alpha$.
Thus $X$ is indecomposable.
\end{proof}

We give a proof of  Proposition \ref{prop_extproj}.

\begin{proof}[Proof of Proposition \ref{prop_extproj}]
(1) If $M \in {}^{\perp_{\mathbb{E}}} \mathcal{F}(\Theta)$, then $0 \to M \xrightarrow{1_{M}} M \dashrightarrow$ is the desired $\mathfrak{s}$-conflation.
Thus we assume $M \not \in {}^{\perp_{\mathbb{E}}} \mathcal{F}(\Theta)$.
Let $j$ be the minimum value satisfying $\mathbb{E}(M,\Theta(j))\neq 0$.
Our assumption implies that $\mathcal{E}$ admits a finite sequence of universal extensions by $\Theta(j)$.
By Lemma \ref{lem_brick-uniex}(2), we obtain an $\mathfrak{s}$-conflation
\begin{align}\label{sconfl}
L_{1} \to Q_{1} \xrightarrow{\beta_{1}} M \overset{\delta}\dashrightarrow
\end{align}
such that $L_{1} \in \mathcal{F}(\Theta(j))$ and $Q_{1} \in {}^{\perp_{\mathbb{E}}} \mathcal{F}(\Theta(j)) \cap \mathcal{E}$.
Applying $\mathcal{C}(-, \Theta(l))$ with $l <j$ to \eqref{sconfl}, we have an exact sequence
\begin{align}
\mathbb{E}(M,\Theta(l))\to \mathbb{E}(Q_{1},\Theta(l))\to \mathbb{E}(L_{1},\Theta(l)). \notag
\end{align}
By $\mathbb{E}(M,\Theta(l))=0=\mathbb{E}(L_{1},\Theta(l))$, we obtain $\mathbb{E}(Q_{1},\Theta(l))=0$. Therefore $Q_{1}\in {}^{\perp_{\mathbb{E}}}\mathcal{F}(\Theta(\le j))$.

We prove that there exists an $\mathfrak{s}$-conflation $K \to P \to M \dashrightarrow$ such that $K \in \mathcal{F}(\Theta(\ge j))$ and $P \in {}^{\perp_{\mathbb{E}}}\mathcal{F}(\Theta)\cap \mathcal{E}$.
If $Q_{1}\in {}^{\perp_{\mathbb{E}}}\mathcal{F}(\Theta(>j))$, then \eqref{sconfl} is the desired $\mathfrak{s}$-conflation.
In the following, we assume $Q_{1}\notin {}^{\perp_{\mathbb{E}}}\mathcal{F}(\Theta(>j))$ and take the minimum value $j_{2}>j_{1}:=j$ satisfying $\mathbb{E}(Q_{1},\Theta(j_{2}))\neq 0$.
By an argument similar to \eqref{sconfl}, we have an $\mathfrak{s}$-conflation $L_{1}'\to Q_{2}\xrightarrow{\beta_{2}} Q_{1}\dashrightarrow$ with $L'_{1}\in \mathcal{F}(\Theta(j_{2}))$ and $Q_{2}\in {}^{\perp_{\mathbb{E}}}\mathcal{F}(\Theta(\le j_{2})) \cap \mathcal{E}$.
Applying (ET4)$^{\op}$ in \cite[Definition 2.12]{NP19} to the $\mathfrak{s}$-conflation and \eqref{sconfl}, we obtain a commutative diagram of $\mathfrak{s}$-conflations
\begin{align}
\xymatrix{
L'_{1} \ar[r] \ar@{=}[d] &L_{2} \ar[r]\ar[d]&L_{1} \ar[d] \ar@{-->}[r] &\\
L'_{1}  \ar[r]&Q_{2}\ar[r]^-{\beta_{2}}\ar[d]&Q_{1} \ar@{-->}[r]\ar[d]^-{\beta_{1}}&.\\
&M\ar@{=}[r]\ar@{-->}[d] &M \ar@{-->}[d]&&\\
&&&
}\notag
\end{align}
In particular, we have the $\mathfrak{s}$-conflation $L_{2}\to Q_{2}\xrightarrow{\beta_{1}\beta_{2}}M\dashrightarrow$ with $L_{2}\in \mathcal{F}(\Theta(\ge j_{1}))$ and $Q_{2}\in {}^{\perp_{\mathbb{E}}}\mathcal{F}(\Theta(\le j_{2})) \cap \mathcal{E}$.
By repeating this procedure, there exists an $\mathfrak{s}$-conflation $L_{m} \to Q_{m} \xrightarrow{\beta_{1} \cdots \beta_{m}} M \dashrightarrow$ such that $L_{m} \in \mathcal{F}(\Theta(\ge j))$ and $Q_{m} \in {}^{\perp_{\mathbb{E}}}\mathcal{F}(\Theta) \cap \mathcal{E}$ for some $m>0$.
Since $\mathcal{C}$ is a Krull--Schmidt category, we can take a right minimal morphism $\alpha_{M}$ of $\beta_{1}\cdots \beta_{m}$.
By Remark \ref{rem_minmor}, we obtain an $\mathfrak{s}$-conflation $K \to P \xrightarrow{\alpha_{M}} M \dashrightarrow$ with  $K\in\add L_{m}$ and $P\in\add Q_{m}$.
By Proposition \ref{prop_dirsummand}(2), $\mathcal{F}(\Theta(\ge j))$ is closed under direct summands.
Thus $K\in\mathcal{F}(\Theta(\ge j))$.
Moreover, since $\mathcal{E}$ is extension-closed, we have $P\in{}^{\perp_{\mathbb{E}}}\mathcal{F}(\Theta)\cap\mathcal{E}$.

(2) In addition, we assume that $\Theta$ satisfies (MS4) and $\mathcal{E}=\mathcal{F}(\Theta)$.
Since $\mathcal{C}$ is a Krull--Schmidt category, $\End_{\mathcal{C}}(M)$ is local.
By Lemma \ref{lem_indec}, it is enough to show that $\mathcal{C}(\alpha_{M},M)$ is an isomorphism.
First we claim that $\mathcal{C}(\beta_{1}, M)$ is an isomorphism.
Applying $\mathcal{C}(-,M)$ to \eqref{sconfl} induces an exact sequence
\begin{align}
\mathbb{E}^{-1}(L_{1},M) \to \mathcal{C}(M,M)\xrightarrow{\mathcal{C}(\beta_{1},M)} \mathcal{C}(Q_{1},M)\to \mathcal{C}(L_{1},M)\xrightarrow{\delta^{\sharp}_{M}} \mathbb{E}(M,M)\notag
\end{align}
by (NE2) and Proposition \ref{prop_longex}.
By $L_{1}, M \in \mathcal{E}=\mathcal{F}(\Theta)$, it follows from (MS4) that $\mathbb{E}^{-1}(L_{1}, M)=0$.
If $M\in \mathcal{F}(\Theta(\ge j))^{\perp_{\mathcal{C}}}$, then $\mathcal{C}(L_{1},M)=0$.
On the other hand, if $M=\Theta(j)$, then the condition  $\mathbb{E}(M,\Theta(j))\neq 0$ implies that $\Theta(j)$ is a brick.
Hence the map $\delta^{\sharp}_{M}$ is an isomorphism by Lemma \ref{lem_brick-uniex}(1-c).
Therefore, for the both cases, $\mathcal{C}(\beta_{1},M)$ is an isomorphism.

Next we show that $\mathcal{C}(\beta_{i},M)$ is an isomorphism for each $i\in [2,m]$.
Applying $\mathcal{C}(-,M)$ to $L'_{i-1}\to Q_{i}\xrightarrow{\beta_{i}}Q_{i-1}\dashrightarrow$, we have an exact sequence
\begin{align}
\mathbb{E}^{-1}(L'_{i-1}, M) \to \mathcal{C}(Q_{i-1},M) \xrightarrow{\mathcal{C}(\beta_{i},M)} \mathcal{C}(Q_{i},M) \to \mathcal{C}(L'_{i-1},M)\notag
\end{align}
by (NE2).
By $L'_{i-1}, M \in \mathcal{E}=\mathcal{F}(\Theta)$, it follows from (MS4) that $\mathbb{E}^{-1}(L'_{i-1}, M)=0$.
Due to our assumption, $M \in \mathcal{F}(\Theta(\ge j))^{\perp_{\mathcal{C}}}$ or $M \in \add \Theta(j)$.
Since $L'_{i-1}\in \mathcal{F}(\Theta(j_{i}))$ with $j_{i}>j$,
we have $\mathcal{C}(L'_{i-1}, M)=0$.
Thus $\mathcal{C}(\beta_{1} \cdots \beta_{m} ,M)$ is an isomorphism, and hence so is $\mathcal{C}(\alpha_{M},M)$.
Therefore $P$ is indecomposable.
\end{proof}

As an application of Proposition \ref{prop_extproj}, we have the following result.

\begin{corollary}
Let $\mathcal{A}$ be a length abelian category and
let $\Theta$ be a mixed standardizable set satisfying that $\add \Theta$ is $\mathbb{E}$-finite.
If $\mod A$ admits a finite sequence of universal extensions by $\Theta$, then $\mathcal{F}(\Theta)$ is covariantly finite.
If in addition, $\mod A$ admits a finite sequence of universal coextensions by $\Theta$, then $\mathcal{F}(\Theta)$ is functorially finite.
\end{corollary}

\begin{proof}
We regard $\mathcal{A}$ as the extriangulated category with the negative first extension $\mathbb{E}^{-1}=0$ by Example \ref{ex_neg}(2).
Let $M \in \mathcal{A}$.
We show that there exists a left $\mathcal{F}(\Theta)$-approximation of $M$.
Since $\mathcal{A}$ is a length abelian category, we can take the rejection $\mathrm{rej}_{\mathcal{F}(\Theta)}(M)$ of $M$.
Let $M':= M/ \mathrm{rej}_{\mathcal{F}(\Theta)}(M)$.
Then we have an exact sequence $0 \to M' \to X \to C \to 0$ with $X \in \mathcal{F}(\Theta)$.
By Proposition \ref{prop_extproj}(1), there exists an exact sequence $0 \to K_{C} \to P_{C} \to C \to 0$ such that $K_{C} \in \mathcal{F}(\Theta)$ and $P_{C}\in {}^{\perp_{\Ext_{\mathcal{A}}^{1}}}\mathcal{F}(\Theta)$.
Taking a pull-back, we obtain a commutative diagram
\begin{align}
\xymatrix{
&&0\ar[d]&0\ar[d]\\
&&K_{C} \ar@{=}[r]\ar[d]&K_{C} \ar[d]  & &\\
0 \ar[r] &M' \ar[r]^-{\varphi}\ar@{=}[d]&E\ar[r]\ar[d]&P_{C} \ar[r]\ar[d]&0&\\
0 \ar[r] &M'\ar[r] &X \ar[r]\ar[d]&C \ar[d] \ar[r]&0.&\\
&&0&0&
}\notag
\end{align}
Since $\mathcal{F}(\Theta)$ is extension-closed, we have $E \in \mathcal{F}(\Theta)$.
Applying $\mathcal{A}(-,N)$ with $N\in\mathcal{F}(\Theta)$ to $0\rightarrow M'\xrightarrow{\varphi}E\rightarrow P_{C}\rightarrow 0$, we obtain an exact sequence
\begin{align}
\mathcal{A}(E,N)\xrightarrow{\mathcal{A}(\varphi,N)}\mathcal{A}(M',N)\rightarrow \Ext_{\mathcal{A}}^{1}(P_{C},N)=0,\notag
\end{align}
where the last equality follows from $P_{C}\in {}^{\perp_{\Ext_{\mathcal{A}}^{1}}}\mathcal{F}(\Theta)$.
This implies that the morphism $\varphi:M'\to E$ is a left $\mathcal{F}(\Theta)$-approximation of $M'$.
Thus the composition map of the canonical epimorphism $M \to M'$ and $\varphi: M' \to E$ is a left $\mathcal{F}(\Theta)$-approximation of $M$.
\end{proof}

\subsection{Mixed stratifying systems}
Let $\Bbbk$ be a field and $\mathcal{C}:=(\mathcal{C},\mathbb{E},\mathfrak{s},\mathbb{E}^{-1})$ a $\Bbbk$-linear Krull--Schmidt extriangulated category with a negative first extension.
In this subsection, we introduce the notion of mixed stratifying systems, which plays a crucial role in Dlab--Ringel's standardization method and a Ringel duality.
Moreover, mixed stratifying systems are a common generalization of stratifying systems \cite{ES03} and proper costratifying systems \cite{MPV11} in an abelian category, $\Theta$-projective systems in a triangulated category \cite{MS16} and $F$-projective systems in an exact category \cite{S19}.

\begin{definition}\label{def_mss}
Let $\Theta:=(\Theta(1), \Theta(2), \ldots, \Theta(t))$ be an ordered set of objects in $\mathcal{C}$.
\begin{itemize}
\item[(1)] Let $\mathbb{P}:=(P(1),P(2),\ldots, P(t))$ be an ordered set of indecomposable objects in $\mathcal{C}$.
A pair $(\Theta, \mathbb{P})$ is called a \emph{mixed stratifying system (of size $t$)} in $\mathcal{C}$ if it satisfies the following conditions.
\begin{itemize}
\item[(MSS1)] For each $i \in [1, t]$, $\Theta(i)$ is a brick or a stone.
\item[(MSS2)] If $i>j$, then $\mathcal{C}(\Theta(i),\Theta(j))=0$.
\item[(MSS3)] For each $i\in [1,t]$, there exists an $\mathfrak{s}$-conflation
\begin{align}
K(i)\to P(i)\xrightarrow{\alpha_{i}} \Theta(i)\dashrightarrow \notag
\end{align}
such that $K(i)\in \mathcal{F}(\Theta(\ge i))$ and $P(i)\in {}^{\perp_{\mathbb{E}}}\mathcal{F}(\Theta)$.
\item[(MSS4)] $\mathbb{E}^{-1}(\Theta,\Theta)=0$.
\end{itemize}
\item[(2)] Let $\mathbb{I}:=(I(1),I(2),\ldots, I(t))$ be an ordered set of  indecomposable objects in $\mathcal{C}$.
A pair $(\Theta, \mathbb{I})$ is called a \emph{mixed costratifying system (of size $t$)} in $\mathcal{C}$ if it satisfies the following conditions.
\begin{itemize}
\item[(MCS1)] For each $i \in [1, t]$, $\Theta(i)$ is a brick or a stone.
\item[(MCS2)] If $i>j$, then $\mathcal{C}(\Theta(i),\Theta(j))=0$.
\item[(MCS3)] For each $i\in [1,t]$, there exists an $\mathfrak{s}$-conflation
\begin{align}
\Theta(i)\to I(i)\to C(i)\dashrightarrow \notag
\end{align}
such that $I(i)\in \mathcal{F}(\Theta)^{\perp_{\mathbb{E}}}$ and $C(i)\in \mathcal{F}(\Theta(\le i))$.
\item[(MCS4)] $\mathbb{E}^{-1}(\Theta,\Theta)=0$.
\end{itemize}
\item[(3)] A triplet $(\Theta, \mathbb{P}, \mathbb{I})$ is called a \emph{mixed bistratifying system} in $\mathcal{C}$ if $(\Theta,\mathbb{P})$ is a mixed stratifying system and $(\Theta,\mathbb{I})$ is a mixed costratifying system.
\end{itemize}
\end{definition}

Note that a mixed costratifying system is a dual notion of a mixed stratifying system.

\begin{remark}\label{rem_op}
A pair $(\Theta, \mathbb{I})$ is a mixed costratifying system in $\mathcal{C}$ if and only if $(\Theta^{\mathrm{op}}, \mathbb{I}^{\mathrm{op}})$ is a mixed stratifying system in the opposite category $\mathcal{C}^{\mathrm{op}}$, where $\Theta^{\mathrm{op}}:=(\Theta'(1),\ldots, \Theta'(t))$ and $\mathbb{I}^{\mathrm{op}}:=(I'(1),\ldots, I'(t))$ with $\Theta'(i):=\Theta(t-i+1)$ and  $I'(i):=I(t-i+1)$.
\end{remark}

In the rest of this subsection, we give only results for mixed stratifying systems.
Now we study basic properties of mixed stratifying systems.
Let $(\Theta,\mathbb{P})$ be a mixed stratifying system in $\mathcal{C}$.
Applying $\mathcal{C}(-, \Theta(i))$ to the $\mathfrak{s}$-conflation in (MSS3) induces an exact sequence
\begin{align}
\mathbb{E}^{-1}(K(i), \Theta(i)) \to \mathcal{C}(\Theta(i), \Theta(i))\xrightarrow{\mathcal{C}(\alpha_{i},\Theta(i))} \mathcal{C}(P(i), \Theta(i)). \notag
\end{align}
By (MSS4), the left-hand side vanishes.
Thus the morphism $\alpha_{i}$ is non-zero.
This implies that $\alpha_{i}$ is right minimal.
Hence $\mathbb{P}$ is uniquely determined, that is, if $(\Theta, \mathbb{P}')$ is also a mixed stratifying system, then $P(i)\cong P'(i)$ for all $i\in[1,t]$, where $\mathbb{P}':=(P'(1),P'(2),\ldots, P'(t))$.
Moreover, we show that $\mathbb{P}$ is a basic object.
Suppose to the contrary that $P(i) \cong P(j)$ with $i>j$.
Then we have an isomorphism $\varphi: P(i) \to P(j)$.
By $\mathcal{C}(K(i), \Theta(j))=0$, there exists a non-zero morphism $\varphi': \Theta(i)\to \Theta(j)$.
This contradicts to (MSS2).
Consequently, $\mathbb{P}$ is basic.

Recall the notion of \emph{projective generators} of an extension-closed  subcategory $\mathcal{X}\subseteq \mathcal{C}$.
Note that $\mathcal{X}$ becomes an extriangulated category by restricting the extriangulated structure of $\mathcal{C}$.
An object $P \in \mathcal{X}$ is said to be \emph{projective} in  $\mathcal{X}$ if $P\in {}^{\perp_{\mathbb{E}}}\mathcal{X}$.
An object $X \in \mathcal{X}$ is called a \emph{generator} of $\mathcal{X}$ if for each $M \in \mathcal{X}$, there exists an $\mathfrak{s}$-conflation $K \to X' \to M \dashrightarrow$ such that $K \in \mathcal{X}$ and $X' \in \add X$.
Dually, we define an \emph{injective cogenerator} of $\mathcal{X}$.

\begin{proposition}\label{prop_mss}
Let $(\Theta, \mathbb{P})$ be a mixed stratifying system in $\mathcal{C}$.
Then the following statements hold.
\begin{itemize}
\item[(1)] If $i>j$, then $\mathbb{E}(\Theta(i),\Theta(j))=0$.
In particular, $\Theta$ is a mixed standardizable set in $\mathcal{C}$.
\item[(2)] $\mathbb{P}$ is a projective generator of $\mathcal{F}(\Theta)$. Moreover, we have $\add \mathbb{P}={}^{\perp_{\mathbb{E}}}\mathcal{F}(\Theta) \cap \mathcal{F}(\Theta)$.
\end{itemize}
\end{proposition}
\begin{proof}
(1) Assume $i >j$.
By (MSS3), we have an $\mathfrak{s}$-conflation $K(i) \rightarrow P(i) \rightarrow \Theta(i) \dashrightarrow$.
Applying $\mathcal{C}(-, \Theta(j))$ to the $\mathfrak{s}$-conflation induces an exact sequence
\begin{align*}
\mathcal{C}(K(i), \Theta(j)) \rightarrow \mathbb{E}(\Theta(i), \Theta(j)) \rightarrow \mathbb{E}(P(i), \Theta(j)).
\end{align*}
The left-hand side vanishes by $K(i) \in \mathcal{F}(\Theta(\ge i))$ and (MSS2).
On the other hand, the right-hand side vanishes by $P(i) \in {}^{\perp_{\mathbb{E}}}\mathcal{F}(\Theta)$.
Thus we have $\mathbb{E}(\Theta(i),\Theta(j))=0$.

(2) First we prove that $\mathbb{P}$ is a projective generator of $\mathcal{F}(\Theta)$.
By (MSS3), $\mathbb{P}$ is projective in $\mathcal{F}(\Theta)$.
We show that $\mathbb{P}$ is a generator of $\mathcal{F}(\Theta)$.
Let $M\in \mathcal{F}(\Theta)$.
By Lemma \ref{lem_filt}, we have a sequence of $\mathfrak{s}$-inflations of $M$
\begin{align}
0=M_{0}\xrightarrow{f_{0}} M_{1}\xrightarrow{f_{1}} M_{2}\to \cdots \xrightarrow{f_{l-1}} M_{l}=M \notag
\end{align}
such that $C_{i}:=\cone (f_{i-1})\in \add \Theta(j_{i})$ for each $i\in [1,l]$.
By induction on $l$, we prove that there exists an $\mathfrak{s}$-conflation $K \to P \to M \dashrightarrow$ such that $K \in \mathcal{F}(\Theta)$ and $P \in \add \mathbb{P}$.
If $l=1$, then $M=M_{1}\in \add \Theta(j_{1})$.
Hence the assertion follows from (MSS3).
Assume $l \ge 2$.
Then we have the $\mathfrak{s}$-conflation $M_{l-1}\xrightarrow{f_{l-1}} M_{l}\to C_{l}\dashrightarrow$.
By (MSS3), there exists an $\mathfrak{s}$-conflation $K'_{l}\to P'_{l}\to C_{l}\dashrightarrow$ with $K'_{l}\in \mathcal{F}(\Theta)$ and $P'_{l}\in \add \mathbb{P}$.
Due to \cite[Proposition 3.15]{NP19}, we have a commutative diagram of $\mathfrak{s}$-conflations
\begin{align}
\xymatrix{
&K'_{l} \ar@{=}[r]\ar[d]&K'_{l} \ar[d] \\
M_{l-1} \ar[r]\ar@{=}[d] &E\ar[r] \ar[d]&P'_{l} \ar@{-->}[r]\ar[d]&\\
M_{l-1} \ar[r] &M_{l}\ar[r]\ar@{-->}[d] &C_{l} \ar@{-->}[d]\ar@{-->}[r]&.\\
&&&
}\notag
\end{align}
By $\mathbb{E}(P'_{l}, M_{l-1})=0$, the middle row of the diagram above splits.
Hence we obtain an $\mathfrak{s}$-conflation $K'_{l} \to M_{l-1}\oplus P'_{l} \to M_{l} \dashrightarrow$.
By the induction hypothesis, there exists an $\mathfrak{s}$-conflation $K_{l-1}\xrightarrow{\beta} P_{l-1}\xrightarrow{\alpha} M_{l-1}\dashrightarrow$ such that $K_{l-1}\in \mathcal{F}(\Theta)$ and $P_{l-1}\in \add \mathbb{P}$.
Applying  (ET4)$^{\op}$ in \cite[Definition 2.12]{NP19}, we have a commutative diagram of $\mathfrak{s}$-conflations
\begin{align}
\xymatrix{
K_{l-1} \ar[r] \ar@{=}[d] &K_{l} \ar[r]\ar[d]&K'_{l} \ar[d] \ar@{-->}[r] &\\
K_{l-1} \ar[r]^-{\tiny\left[\begin{smallmatrix}\beta\\ 0 \end{smallmatrix}\right]}&P_{l-1} \oplus P'_{l}\ar[r]^-{\tiny\left[\begin{smallmatrix}\alpha &0\\ 0 & 1\end{smallmatrix}\right]}\ar[d]&M_{l-1} \oplus P'_{l} \ar@{-->}[r]\ar[d]&.\\
&M_{l}\ar@{=}[r]\ar@{-->}[d] &M_{l} \ar@{-->}[d]&&\\
&&&
}\notag
\end{align}
Since $\mathcal{F}(\Theta)$ is extension-closed, we obtain $K_{l} \in \mathcal{F}(\Theta)$.
Therefore the middle column of the diagram above is the desired $\mathfrak{s}$-conflation.

Next we prove $\add \mathbb{P}={}^{\perp_{\mathbb{E}}}\mathcal{F}(\Theta) \cap \mathcal{F}(\Theta)$.
It is enough to show ${}^{\perp_{\mathbb{E}}}\mathcal{F}(\Theta) \cap \mathcal{F}(\Theta) \subseteq \add \mathbb{P}$.
Let $X \in{}^{\perp_{\mathbb{E}}}\mathcal{F}(\Theta) \cap  \mathcal{F}(\Theta)$.
Since $\mathbb{P}$ is a projective generator of $\mathcal{F}(\Theta)$, we have an $\mathfrak{s}$-conflation $K \to P \to X \dashrightarrow$ with $K \in \mathcal{F}(\Theta)$ and $P \in \add \mathbb{P}$.
By $X \in {}^{\perp_{\mathbb{E}}}\mathcal{F}(\Theta)$, the  $\mathfrak{s}$-conflation splits.
Hence we have $X \in \add \mathbb{P}$.
\end{proof}

Let $A$ be a basic finite dimensional $\Bbbk$-algebra with $n$ simple modules (up to isomorphisms).
By Proposition \ref{prop_mss}, we give a characterization of $\mathcal{F}(\Theta)$ to be a resolving subcategory of $\mod A$.
An extension-closed subcategory $\mathcal{X}$ of $\mod A$ is called a \emph{resolving subcategory} if $A \in \mathcal{X}$ and $\mathcal{X}$ is closed under direct summands and epikernels.

\begin{proposition}\label{prop_resol}
Let $(\Theta, \mathbb{P})$ be a mixed stratifying system of size $n$ in $\mod A$.
Then the following statements are equivalent.
\begin{itemize}
\item[(1)] $\mathcal{F}(\Theta)$ is a resolving subcategory of $\mod A$.
\item[(2)] $A \in \mathcal{F}(\Theta)$
\item[(3)] $\mathbb{P}\cong A$ $($as objects$)$.
\end{itemize}
\end{proposition}

\begin{proof}
We regard $\mod A$ as the extriangulated category with the negative first extension $\mathbb{E}^{-1}=0$ by Example \ref{ex_neg}(2).

(1)$\Rightarrow$(2): This is clear.

(2)$\Rightarrow$(3): By Proposition \ref{prop_mss}(2), $\add \mathbb{P}={}^{\perp_{\Ext_{A}^{1}}}\mathcal{F}(\Theta) \cap \mathcal{F}(\Theta)$.
Hence $A\in \mathcal{F}(\Theta)$ implies $A\in \add \mathbb{P}$.
Since $A$ and $\mathbb{P}$ are basic, we have $\mathbb{P}\cong A$.

(3)$\Rightarrow$(1):
Since $(\Theta, \mathbb{P})$ is a mixed stratifying system, $\mathcal{F}(\Theta)$ is extension-closed and closed under direct summands by Propositions \ref{prop_dirsummand}(2) and \ref{prop_mss}(1).
By $A\cong \mathbb{P} \in \mathcal{F}(\Theta)$, it is enough to prove that $\mathcal{F}(\Theta)$ is closed under epikernels.
Let $\varphi: X \to Y$ be an epimorphism with $X, Y \in \mathcal{F}(\Theta)$.
By Proposition \ref{prop_mss}(2), we have an exact sequence $0 \to K_{Y} \to P_{Y} \to Y \to 0$ with $K_{Y} \in \mathcal{F}(\Theta)$ and $P_{Y} \in \add \mathbb{P}=\add A$.
Taking a pull-back, we obtain a commutative diagram
\begin{align}
\xymatrix{
&&0\ar[d]&0\ar[d]\\
&&K_{Y} \ar@{=}[r]\ar[d]&K_{Y} \ar[d] & &\\
0 \ar[r] &\ker \varphi \ar[r]\ar@{=}[d]& E\ar[r]\ar[d]&P_{Y}  \ar[r]\ar[d]&0&\\
0 \ar[r] &\ker \varphi \ar[r] &X  \ar[r]\ar[d]&Y  \ar[d]\ar[r]&0.&\\
&&0&0&
}\notag
\end{align}
Since $\mathcal{F}(\Theta)$ is extension-closed, we have $E \in \mathcal{F}(\Theta)$.
By $P_{Y} \in \add A$, the middle row of the diagram above splits.
Hence $\ker \varphi \in \add E\subseteq \mathcal{F}(\Theta)$.
This completes the proof.
\end{proof}

We construct mixed stratifying systems from mixed standardizable sets.

\begin{theorem}\label{mainthm1}
Let $\Theta:=(\Theta(1), \ldots, \Theta(t))$ be a mixed standardizable set in $\mathcal{C}$ satisfying that $\add \Theta$ is $\mathbb{E}$-finite.
If $\mathcal{F}(\Theta)$ admits a finite sequence of universal extensions by $\Theta$, then the following statements hold.
\begin{itemize}
\item[(1)] For each $i\in [1,t]$, there exists an $\mathfrak{s}$-conflation $K(i) \to P(i) \xrightarrow{\alpha_{i}} \Theta(i) \dashrightarrow$ such that $K(i) \in \mathcal{F}(\Theta(\ge i))$, $P(i)$ is an indecomposable projective object in $\mathcal{F}(\Theta)$ and $\alpha_{i}$ is right minimal.
\item[(2)] Let $\mathbb{P}:=(P(1),\ldots, P(t))$. Then $(\Theta,\mathbb{P})$ is a mixed stratifying system in $\mathcal{C}$.
\end{itemize}
\end{theorem}

\begin{proof}
(1) Fix an integer $i\in [1,t]$.
If $\Theta(i)\in{}^{\perp_{\mathbb{E}}}\mathcal{F}(\Theta)$, then the assertion follows from Proposition \ref{prop_extproj}(1).
In the following, we assume $\Theta(i)\not\in{}^{\perp_{\mathbb{E}}}\mathcal{F}(\Theta)$.
By Proposition \ref{prop_extproj}(1), there exists an $\mathfrak{s}$-conflation $K(i) \to P(i) \xrightarrow{\alpha_{i}} \Theta(i) \dashrightarrow$ such that $K(i)\in \mathcal{F}(\Theta(\ge j))$, $P(i)\in {}^{\perp_{\mathbb{E}}}\mathcal{F}(\Theta) \cap \mathcal{F}(\Theta)$ and $\alpha_{i}$ is right minimal, where $j$ is the minimum value satisfying $\mathbb{E}(\Theta(i),\Theta(j))\neq 0$.
It follows from (MS3) that $j \ge i$.
We show that $P(i)$ is indecomposable.
If $\mathbb{E}(\Theta(i),\Theta(i))=0$, then $j \neq i$, and hence $\Theta(i)\in \mathcal{F}(\Theta(\ge j))^{\perp_{\mathcal{C}}}$ by (MS2).
If otherwise, we obtain $\Theta(i)=\Theta(j)$.
Hence the assertion follows from Proposition \ref{prop_extproj}(2).

(2) (MSS1), (MSS2) and (MSS4) follow from (MS1), (MS2) and (MS4) respectively.
(MSS3) follows from (1).
Hence $(\Theta, \mathbb{P})$ is a mixed stratifying system.
\end{proof}

As an application of Theorem \ref{mainthm1}, we have the following result.

\begin{corollary}\label{cor_mainthm1}
Let $A$ be a finite dimensional $\Bbbk$-algebra and $(e_{1},e_{2},\ldots, e_{n})$ a complete ordered set of primitive orthogonal idempotents of $A$.
Let $\Theta:=(\Theta(1),\Theta(2),\ldots,\Theta(n))$ with $\Theta(i)\in \{ \Delta(i), \overline{\Delta}(i)\}$.
Then there exists a projective object $\mathbb{P}$ in $\mathcal{F}(\Theta)$ such that $(\Theta, \mathbb{P})$ is a mixed stratifying system in $\mod A$.
If in addition, $\Theta(i)=\Delta(i)$ holds for all $i \in [1, n]$, then there exists an injective object $\mathbb{I}$ in $\mathcal{F}(\Theta)$ such that $(\Theta, \mathbb{P}, \mathbb{I})$ is a mixed  bistratifying system in $\mod A$.
\end{corollary}

\begin{proof}
We regard $\mod A$ as the extriangulated category with the negative first extension $\mathbb{E}^{-1}=0$ by Example \ref{ex_neg}(2).
By Proposition \ref{prop_mixstd}, $\Theta=(\Theta(1), \ldots, \Theta(n))$ is a mixed standardizable set and
$\mod A$ admits a finite sequence of universal extensions by $\Theta$.
Thus, by Theorem \ref{mainthm1}(2), there exists a projective object $\mathbb{P}$ in $\mathcal{F}(\Theta)$ such that $(\Theta, \mathbb{P})$ is a mixed stratifying system in $\mod A$.
Assume $\Theta(i)=\Delta(i)$ for all $i \in [1,n]$.
Since each $\Theta(i)$ is a stone, $\mod A$ admits a finite sequence of universal coextensions by $\Theta$.
By the dual statement of Theorem \ref{mainthm1}, there exists an injective object $\mathbb{I}$ in $\mathcal{F}(\Theta)$ such that $(\Theta, \mathbb{I})$ is a mixed costratifying system in $\mod A$.
This completes the proof.
\end{proof}

By Proposition \ref{prop_brick-univex}(i) and Theorem \ref{mainthm1}, we give an example of a mixed stratifying system.

\begin{example}\label{ex_mbst}
Let $A$ be a finite dimensional $\Bbbk$-algebra and let $\Theta$ be a mixed standardizable set in $\mod A$.
Assume that there exists an integer $l>0$ such that, for each indecomposable $A$-module $M \in \mathcal{F}(\Theta)$, the length of $M$ is bounded by $l$.
Then we have a mixed stratifying system $(\Theta, \mathbb{P})$ by Proposition \ref{prop_brick-univex}(i) and Theorem \ref{mainthm1}(2).
Moreover, there exists a mixed costratifying system $(\Theta, \mathbb{I})$ by the dual statements of Proposition \ref{prop_brick-univex}(i) and Theorem \ref{mainthm1}(2).
Therefore we obtain a mixed bistratifying system $(\Theta, \mathbb{P}, \mathbb{I})$.
In particular, we can recover \cite[Theorem 3.11]{MPV14}.
\end{example}

\section{Standardization and Ringel duality for mixed stratified algebras}
In this section, we construct mixed stratified algebras following Dlab--Ringel's standardization method and study Ringel duality of mixed stratified algebras from the viewpoint of mixed stratifying systems.

\subsection{Mixed stratified algebras}
In this subsection, we collect properties of mixed stratified algebras introduced by \'Agoston, Dlab and Luk\'acs \cite{ADL98}.
First let us recall the definition of mixed stratified algebras.
Let $A$ be a basic finite dimensional $\Bbbk$-algebra.
Fix a complete ordered set $(e_{1},e_{2},\ldots, e_{n})$ of primitive orthogonal idempotents of $A$.
We regard $A$ as an ordered set $(e_{1}A, e_{2}A, \ldots, e_{n}A)$.
For each $i \in [1, n]$, let $\Delta(i)$ be the $i$-th standard module and $\overline{\Delta}(i)$ the $i$-th proper standard module. 
Let $\Delta:=(\Delta(1),\Delta(2),\ldots,\Delta(n))$ denote the ordered set of standard modules and $\overline{\Delta}:=(\overline{\Delta}(1),\overline{\Delta}(2),\ldots,\overline{\Delta}(n))$ denote the ordered set of proper standard modules.

\begin{definition}[{\cite[Definition 1.3]{ADL98}}]
Let $A$ be a finite dimensional $\Bbbk$-algebra and let  $\Theta:=(\Theta(1), \ldots, \Theta(n))$ be an ordered set of $A$-modules.
We call $A$ a \emph{mixed stratified algebra} with respect to $\Theta$ if $A \in \mathcal{F}(\Theta)$ and $\Theta(i)\in \{ \Delta(i), \overline{\Delta}(i)\}$ for each $i\in [1,n]$.
\end{definition}

We call $A$ a \emph{standardly stratified algebra} with respect to $\Delta$ if $A \in \mathcal{F}(\Delta)$. For details, see \cite{CPS96, D96}.
A typical example of a mixed stratified algebra is a standardly stratified algebra.
We give a concrete example of mixed stratified algebras.

\begin{example}
Assume that $\Bbbk$ is an algebraically closed field.
Let $A$ be the $\Bbbk$-algebra defined by the quiver
\begin{align}
\xymatrix{1\ar@<0.5ex>[r]^-{a} & 2 \ar@<0.5ex>[r]^-{b} \ar@<0.5ex>[l]^-{c} &3 \ar@<0.5ex>[l]^-{d}
 }\notag
\end{align}
with relations $ab, dc, aca$ and $dbd$.
Then $\Delta(1)=\overline{\Delta}(1)=\top e_{1}A$, $\Delta(2)=\cok (e_{3}A\xrightarrow{b\cdot -} e_{2}A)$, $\Delta(3)=e_{3}A$, $\overline{\Delta}(2)=\cok (\Delta(2)\xrightarrow{ca\cdot-}\Delta(2))$, and $\overline{\Delta}(3)=e_{3}A/\soc e_{3}A$.
We put $\Theta(1):=\Delta(1)$, $\Theta(2):=\overline{\Delta}(2)$, and $\Theta(3):=\Delta(3)$.
Then $A$ is a mixed stratified algebra with respect to $\Theta$ but not a standardly stratified algebra with respect to $\Delta$.
\end{example}

We describe the left-right symmetry of mixed stratified algebras. 
Let $A$ be a finite dimensional $\Bbbk$-algebra and let $\mathsf{D}:=\Hom_{\Bbbk}(-,\Bbbk)$. 
For each $i \in [1, n]$, let $\nabla(i)$ denote the $i$-th costandard module and $\overline{\nabla}(i)$ the $i$-th proper costandard module. 
Let $\Xi(i)\in\{\nabla(i), \overline{\nabla}(i)\}$ for each $i\in\{1,n\}$. 
Since $\mathsf{D}\nabla(i)$ is the $i$-th standard left $A$-module and $\mathsf{D}\overline{\nabla}(i)$ is the $i$-th proper standard left $A$-module, $\mathsf{D}\Xi:=(\mathsf{D}\Xi(1),\ldots, \mathsf{D}\Xi(n))$ is a mixed standardizable set in $\mod A^{\op}$ by Proposition \ref{prop_mixstd}. 

\begin{remark}[{\cite[Lemma 2.2]{D96}, \cite[Theorem 2.2]{ADL98}}]\label{rem_mixstr}
An algebra $A$ is a mixed stratified algebra with respect to $\Theta:=(\Theta(1),\ldots,\Theta(n))$ if and only if $A^{\op}$ is a mixed stratified algebra with respect to $\mathsf{D}\Xi:=(\mathsf{D}\Xi(1), \ldots, \mathsf{D}\Xi(n))$, where
\begin{align}
\Xi(i):=\left\{
\begin{array}{ll}
\overline{\nabla}(i) & \textnormal{if } \Theta(i)=\Delta(i)\\ 
\nabla(i) & \textnormal{if } \Theta(i)=\overline{\Delta}(i). \notag
\end{array}
\right.
\end{align}
\end{remark}

In the following, we regard $\mod A$ as the extriangulated category with the negative first extension $\mathbb{E}^{-1}=0$ by Example \ref{ex_neg}(2).
The aim of this subsection is to give characterizations of mixed stratified algebras in terms of mixed stratifying systems. 

\begin{theorem}\label{thm_mixed-bistr}
Let $A$ be a finite dimensional $\Bbbk$-algebra and let $\Theta:=(\Theta(1), \ldots, \Theta(n))$ be an ordered set of $A$-modules.
Then the following statements  are equivalent.
\begin{itemize}
\item[(1)] $A$ is a mixed stratified algebra with respect to $\Theta$.
\item[(2)] The pair $(\Theta, A)$ is a mixed stratifying system in $\mod A$.
\item[(3)] There exists an ordered set $\mathbb{T}:=(T(1), \ldots, T(n))$ of $A$-modules such that $(\Theta, A, \mathbb{T})$ is a mixed bistratifying system. 
\end{itemize}
\end{theorem}

To show Theorem \ref{thm_mixed-bistr}, we need the following results.

\begin{lemma}\label{lem_pstandard}
Let $A$ be a finite dimensional $\Bbbk$-algebra.
Assume that $(\Theta, A)$ is a mixed stratifying system in $\mod A$.
Then the following statements hold.
\begin{itemize}
\item[(1)] If $\Theta(i)$ is a stone, then $\Theta(i)=\Delta(i)$.
\item[(2)] If $\Theta(i)$ is a brick, then $\Theta(i)=\overline{\Delta}(i)$.
\end{itemize}
\end{lemma}
\begin{proof}
By (MSS3), we have an exact sequence $0 \to K(i) \rightarrow P(i) \xrightarrow{\alpha_{i}} \Theta(i) \to 0$ with $K(i) \in \mathcal{F}(\Theta(\ge i))$ and $P(i)=e_{i}A$.
Applying $\Hom_{A}(-,X)$ to the exact sequence induces an exact sequence
\begin{align}
0 \to \Hom_{A}(\Theta(i), X) \xrightarrow{\Hom(\alpha_{i},X)} \Hom_{A}(P(i),X) \rightarrow\Hom_{A}(K(i), X) \xrightarrow{\delta_{X}}\Ext_{A}^{1}(\Theta(i), X) \to 0.   \notag
\end{align}
We show that
\begin{itemize}
\item[(i)] $\Theta(i) \in \mathcal{F}(S(1) \oplus \cdots \oplus S(i))$,
\item[(ii)] there exists no epimorphism $f: K(i) \to \Theta(i)^{\oplus l}$ such that $l \neq 0$ and $\delta_{\Theta(i)^{\oplus l}}(f)=0$.
\end{itemize}
First we claim (i).
Assume $j>i$.
Applying $\Hom_{A}(-, \Theta(i))$ to the exact sequence $0 \to K(j) \to P(j) \xrightarrow{\alpha_{j}} \Theta(j) \to 0$ induces an exact sequence
\begin{align}
\Hom_{A}(\Theta(j), \Theta(i)) \to \Hom_{A}(P(j), \Theta(i)) \to \Hom_{A}(K(j), \Theta(i)).\notag
\end{align}
Since the left-hand side and right-hand side vanish by (MSS2), we have $\Hom_{A}(P(j), \Theta(i))=0$.
Hence $\Theta(i) \in \mathcal{F}(S(1) \oplus \cdots \oplus S(i))$.
Next we claim (ii).
Suppose to the contrary that there exists an epimorphism $f: K(i) \to \Theta(i)^{\oplus l}$ such that $l \neq 0$ and $\delta_{\Theta(i)^{\oplus l}}(f)=0$.
Taking a push-out, we have a commutative diagram
\begin{align}
\xymatrix{
0\ar[r]& K(i) \ar[r]\ar[d]_-{f} &P(i) \ar[r]^{\alpha_{i}}\ar[d]& \Theta(i) \ar@{=}[d]\ar[r]&0&\\
0\ar[r]& \Theta(i)^{\oplus l} \ar[r]\ar[d] &E \ar[r]\ar[d]&\Theta(i) \ar[r]&0.&\\
&0 &0&&
}\notag
\end{align}
Since $\delta_{\Theta(i)^{\oplus l}}(f)=0$, we obtain $E \cong \Theta(i)^{\oplus l+1}$.
Thus $e_{i}A=P(i) \to \Theta(i)^{\oplus l+1}$ is an epimorphism, a contradiction.

(1) Assume that $\Theta(i)$ is a stone.
By $K(i) \in \mathcal{F}(\Theta(\ge i))$, there exists an exact sequence $0 \to K' \to K(i) \xrightarrow{f} \Theta(i)^{\oplus d}\to 0$ with $K' \in \mathcal{F}(\Theta(>i))$ and $d \ge0$.
Since $\Theta(i)$ is a stone, we have $\delta_{\Theta(i)^{\oplus d}}(f)=0$.
By (ii), we obtain $d=0$, and hence $K(i) \cong K' \in \mathcal{F}(\Theta(>i))$.
This implies $\Hom_{A}(K(i), S(j))=0$ for all $j \le i$.
Therefore $\Theta(i)$ is the $i$-th standard module by (i).

(2) Assume that $\Theta(i)$ is a brick.
First we show that $\Hom(\alpha_{i},\Theta(i)): \Hom_{A}(\Theta(i), \Theta(i)) \to \Hom_{A}(P(i), \Theta(i))$ is an isomorphism.
By $K(i) \in \mathcal{F}(\Theta(\ge i))$, we have an exact sequence $0 \to K' \to K(i) \xrightarrow{p} K''\to 0$ with $K' \in \mathcal{F}(\Theta(>i))$ and $K''\in\mathcal{F}(\Theta(i))$.
Then $K''\neq0$ if and only if $\Hom_{A}(K(i),\Theta(i))\neq 0$.
Thus there is nothing to prove if $K''=0$.
In the following, we assume $K''\neq0$.
Let $f\neq 0$ be an arbitrary morphism in $\Hom_{A}(K(i),\Theta(i))$.
By $\Hom_{A}(K', \Theta(i))=0$, there exists a non-zero morphism $\varphi: K'' \to \Theta(i)$ such that $\varphi p=f$.
Since $K''\in\mathcal{F}(\Theta(i))$ and $\Theta(i)$ is a brick, $\varphi$ is an epimorphism, and hence so is $f$.
By (ii), $\delta_{\Theta(i)}(f)\neq 0$.
Thus $\delta_{\Theta(i)}$ is a monomorphism.
This implies that $\Hom(\alpha_{i},\Theta(i))$ is an isomorphism.

Next we show $\Hom_{A}(P(i),\rad\Theta(i))=0$.
Let $g: P(i)\to\rad \Theta(i)$ be an arbitrary morphism and $\iota:\rad \Theta(i) \to\Theta(i)$ the inclusion.
Then $\iota g\in\Hom_{A}(P(i),\Theta(i))$.
Since $\Hom(\alpha_{i},\Theta(i))$ is an isomorphism, there uniquely exists $h\in\End_{A}(\Theta(i))$ such that $\iota g=h\alpha_{i}$.
Since $\End_{A}(\Theta(i))$ is a division ring, the morphism  $h$ is either zero or an isomorphism.
In the latter case, the morphism $\iota$ is an epimorphism, a contradiction.
Thus $h$ is zero, and hence so is $g$.
This implies $\Hom_{A}(P(i),\rad\Theta(i))=0$.
By (i), we have $\rad \Theta(i) \in \mathcal{F}(S(1) \oplus \cdots \oplus S(i-1))$.
By definition of standard modules, we obtain an exact sequence $0\to K'(i)\to e_{i}A \to \Delta(i)\to0$ with $\top K'(i)\in \add \oplus_{j >i}S(j)$.
It follows from (i) that  $\Hom_{A}(K'(i),\Theta(i))=0$.
Thus we have a commutative diagram
\begin{align}
\xymatrix{
0\ar[r]& K'(i) \ar[r]\ar[d]_-{\gamma'} &e_{i}A \ar[r]\ar@{=}[d]& \Delta(i) \ar[d]_-{\gamma}\ar[r]&0&\\
0\ar[r]& K(i) \ar[r]&e_{i}A \ar[r]&\Theta(i) \ar[r]&0.&
}\notag
\end{align}
Hence $\gamma$ is an epimorphism and $\ker \gamma \cong \cok \gamma'$.
Since $\Hom_{A}(K(i),S(j))=0$ holds for each $j< i$, we have $\Hom_{A}(\ker \gamma, S(j))\cong \Hom_{A}(\cok \gamma', S(j))=0$.
Thus $\Theta(i)$ is the maximal factor module of $\Delta(i)$ such that $\rad \Theta(i) \in \mathcal{F}(S(1) \oplus \cdots \oplus S(i-1))$.
Therefore $\Theta(i)$ is the $i$-th proper standard module.
\end{proof}

Assume that $A$ is a mixed stratified algebra with respect to $\Theta$. Let 
\begin{align}
\Xi(i):=\left\{
\begin{array}{ll}
\overline{\nabla}(i) & \textnormal{if } \Theta(i)=\Delta(i)\\ 
\nabla(i) & \textnormal{if } \Theta(i)=\overline{\Delta}(i). \notag
\end{array}
\right.
\end{align}
Then $\Xi:=(\Xi(n), \ldots, \Xi(1))$ is a mixed standardizable set in $\mod A$ by a dual statement of Proposition \ref{prop_mixstd}.  
It follows from \cite[Theorem 3.11]{BS} that $\mathcal{F}(\Theta)^{\perp_{\Ext_{A}^{1}}}=\mathcal{F}(\Xi)$.
Brundan and Stroppel gave the following result. 

\begin{proposition}\cite[Theorem 4.2]{BS} \label{prop-BS}
For each $i \in[1,n]$, there is an indecomposable module $T(i)$ satisfying the following properties. 
\begin{itemize}
\item[(1)] There is an exact sequence $0 \to \Theta(i) \to  T(i) \to C(i) \to 0$ such that $C(i) \in \mathcal{F}(\Theta(\le i))$. 
\item[(2)] There is an exact sequence $0 \to K(i) \to  T(i) \to \Xi(i) \to 0$ such that $K(i) \in \mathcal{F}(\Xi(\le i))$.
\item[(3)] $T(i)e_{i}$ is an indecomposable $e_{i}Ae_{i}$-module. 
\end{itemize}
In particular, $(\Theta, A, \mathbb{T})$ is a mixed bistratifying system, where $\mathbb{T}:=(T(1), \ldots, T(n))$.  
\end{proposition}

For the convenience of the readers, we give a proof. 
\begin{proof}
We show the assertion by induction on $n$. 
Assume $n=1$. Then this is clear. 

Assume $n \ge 2$. 
Let $\varepsilon_{2}:=e_{2}+\cdots+ e_{n}$. 
By \cite[Theorem 3.18]{BS}, $\varepsilon_{2}A\varepsilon_{2}$ is a mixed stratified algebra with respect to $\Theta\varepsilon_{2}:=(\Theta(2)\varepsilon_{2}, \ldots, \Theta(n)\varepsilon_{2})$ and $\Xi\varepsilon_{2}:=(\Xi(2)\varepsilon_{2}, \ldots, \Xi(n)\varepsilon_{2})$.
By the induction hypothesis, for each $i \in [2,n]$, there is an indecomposable $\varepsilon_{2}A\varepsilon_{2}$-module $T'(i)$ satisfying (1$'$) there is an exact sequence $0 \to \Theta(i)\varepsilon_{2} \to  T'(i) \to C'(i) \to 0$ such that $C'(i) \in \mathcal{F}(\Theta(\le i)\varepsilon_{2})$, (2$'$) there is an exact sequence $0 \to K'(i) \to  T'(i) \to \Xi (i)\varepsilon_{2} \to 0$ such that $K'(i) \in \mathcal{F}(\Xi (\le i)\varepsilon_{2})$ and (3$'$) $T'(i)e_{i}$ is an indecomposable $e_{i}Ae_{i}$-module. 

(i) We assume that $\Theta(1)=\Delta(1)$. 
It follows from \cite[Theorem 3.18(6)]{BS} that $\Theta(i)\varepsilon_{2} \otimes_{\varepsilon_{2}A\varepsilon_{2}}\varepsilon_{2}A\cong \Theta(i)$ and the functor $-\otimes_{\varepsilon_{2}A\varepsilon_{2}}\varepsilon_{2}A$ sends exact sequences in $\mathcal{F}(\Theta\varepsilon_{2})$ to exact sequences in $\mathcal{F}(\Theta)$. 
Applying $-\otimes_{\varepsilon_{2}A\varepsilon_{2}}\varepsilon_{2}A$ to the exact sequence $0 \to \Theta(i)\varepsilon_{2} \to T'(i) \to C'(i) \to 0$ induces an exact sequence 
\begin{align}
0 \to \Theta(i) \to T'(i)\otimes_{\varepsilon_{2}A\varepsilon_{2}}\varepsilon_{2}A \to C'(i) \otimes_{\varepsilon_{2}A\varepsilon_{2}}\varepsilon_{2}A \to 0 \notag
\end{align}
with $C'(i) \otimes_{\varepsilon_{2}A\varepsilon_{2}}\varepsilon_{2}A \in \mathcal{F}(\Theta(\le i))$. 
Let $T_{0}:=T'(i) \otimes_{\varepsilon_{2}A\varepsilon_{2}}\varepsilon_{2}A$ and $d:=\dim \Ext_{A}^{1}(\Theta(1), T_{0})$. 
Taking the ($\Bbbk$-)universal coextension of $T_{0}$ by $\Theta(1)$, we have an exact sequence 
\begin{align}\label{seq-univ-sect4}
0 \to T_{0} \to T_{1}\to \Theta(1)^{\oplus d} \to 0
\end{align}
in $\mathcal{F}(\Theta(\le i))$. 
Note that if $d=0$, then $T_{0}\cong T_{1}$ holds.
Applying $\Hom_{A}(e_{i}A,-)$ to \eqref{seq-univ-sect4}, we have isomorphisms $T_{1}e_{i} \cong T_{0}e_{i} \cong T'(i)e_{i}$.
By the condition (3$'$), $T_{1}e_{i}$ is indecomposable.
Thus there exists an indecomposable direct summand $T_{1}'$ of $T_{1}$ such that $T_{1}e_{i}\cong T_{1}'e_{i}$.
We show that an indecomposable module $T(i):=T'_{1}$ satisfies the conditions (1), (2) and (3) for each $i\ge 2$.
It follows from the construction that $T(i)$ satisfies the condition (3).
By Proposition \ref{prop_dirsummand}(2) and $T_{1}\in \mathcal{F}(\Theta(\leq i))$, we obtain $T(i)\in \mathcal{F}(\Theta(\leq i))$. 
Since $\Hom_{A}(e_{i}A, T(i))\cong T(i)e_{i}\neq 0$ holds, $T(i)$ contains $S(i)$ as a composition factor.
Hence we have $T(i)\notin \mathcal{F}(\Theta(<i))$.
By Proposition \ref{prop_dirsummand}(1), there is an exact sequence 
\begin{align}
0 \to \Theta(i) \to T(i) \to C(i)\to 0\notag
\end{align}
such that $C(i)\in\mathcal{F}(\Theta(\le i))$. 
Therefore $T(i)$ satisfies the condition (1).
In the following, we show that $T(i)$ satisfies the condition (2).
Applying $\Hom_{A}(\Theta(l),-)$ with $l \ge 2$ to \eqref{seq-univ-sect4} induces an exact sequence 
\begin{align}
\Ext^{1}_{A}(\Theta(l),T_{0})\to \Ext^{1}_{A}(\Theta(l),T_{1}) \to \Ext_{A}^{1}(\Theta(l),\Theta(1)^{\oplus d})=0,  \notag
\end{align}
where the last equality follows from Proposition \ref{prop_mixstd}.
By \cite[Theorem 3.18(5)]{BS}, we obtain 
\begin{align}
\Ext_{A}^{1}(\Theta(l), T_{0}) \cong \Ext_{\varepsilon_{2}A\varepsilon_{2}}^{1}(\Theta(l)\varepsilon_{2}, T_{0}\varepsilon_{2})\cong \Ext_{\varepsilon_{2}A\varepsilon_{2}}^{1}(\Theta(l) \varepsilon_{2}, T'(i))=0, \notag
\end{align}  
where the last equality follows from $T'(i)\in \mathcal{F}(\Xi\varepsilon_{2})=\mathcal{F}(\Theta\varepsilon_{2})^{\perp_{\Ext_{A}^{1}}}$.
Thus $ \Ext^{1}_{A}(\Theta(l),T_{1})=0$. 
Applying $\Hom_{A}(\Theta(1),-)$ to \eqref{seq-univ-sect4} induces an exact sequence 
\begin{align}
\Hom_{A}(\Theta(1), \Theta(1)^{\oplus d}) \xrightarrow{\delta_{\Theta(1)}} \Ext^{1}_{A}(\Theta(1), T_{0}) \to \Ext_{A}^{1}(\Theta(1), T_{1})\to \Ext_{A}^{1}(\Theta(1),\Theta(1)^{\oplus d})=0, \notag
\end{align}
where the last equality follows from $\Theta(1)=\Delta(1)$.
Since $\delta_{\Theta(1)}$ is an epimorphism, we obtain $\Ext_{A}^{1}(\Theta(1), T_{1})=0$. 
Thus $T(i)\in\mathcal{F}(\Theta)^{\perp_{\Ext_{A}^{1}}}=\mathcal{F}(\Xi)$. 
By $T(i)\in\mathcal{F}(\Theta(\le i))$, we obtain $T(i)\in\mathcal{F}(S(1)\oplus \cdots \oplus S(i))$, and hence $T(i)\in\mathcal{F}(\Xi(\le i))$.
Since $T(i)$ contains $S(i)$ as a composition factor, we have $T(i)\notin\mathcal{F}(\Xi(<i))$.
Thus it follows from Proposition \ref{prop_dirsummand}(1) that there is an exact sequence 
\begin{align}
0 \to K(i) \to T(i) \to \Xi(i)\to 0 \notag
\end{align}
such that $K(i)\in\mathcal{F}(\Xi(\le i))$. 
Hence $T(i)$ is the desired module for each $i\ge 2$.
Moreover, we put $T(1):=\Delta(1)$.
By $\Theta(1)=\Delta(1)$ and $\Xi(1)=\overline{\nabla}(1)=S(1)$, $T(1)$ clearly satisfies the conditions (1) and (2).
By \cite[Theorem 3.18(1)]{BS}, $T(1)e_{1}=\Delta(1)e_{1}$ is a standard module of $e_{1}Ae_{1}$. 
Therefore $T(1)$ satisfies the condition (3). 

(ii) We assume that $\Theta(1)=\overline{\Delta}(1)$. 
Let $\mathsf{D}:=\Hom_{\Bbbk}(-,\Bbbk)$.
Then $\mathsf{D}\Xi(1)$ is a standard left $A$-module. 
It follows from Remark \ref{rem_mixstr} that $A^{\op}$ is a mixed stratified algebra with respect to $\mathsf{D}\Xi=(\mathsf{D}\Xi(1), \ldots, \mathsf{D}\Xi(n))$. 
By (i), we have an indecomposable left $A$-module $T^{\circ}(i)$ satisfying (1)$^{\circ}$ there is an exact sequence $0\to\mathsf{D}\Xi(i) \to T^{\circ}(i)\to C^{\circ}(i)\to 0$ such that $C^{\circ}(i)\in\mathcal{F}(\mathsf{D}\Xi(\le i))$, (2)$^{\circ}$ there is an exact sequence $0\to K^{\circ}(i) \to T^{\circ}(i)\to\mathsf{D}\Theta(i)\to 0$ such that $K^{\circ}(i)\in\mathcal{F}(\mathsf{D}\Theta(\le i))$ and (3)$^{\circ}$ $e_{i}T^{\circ}(i)$ is an indecomposable left $e_{i}Ae_{i}$-module. 
Thus $\mathsf{D}T^{\circ}(i)$ is a desired module. 
\end{proof}

Now we are ready to prove Theorem \ref{thm_mixed-bistr}.

\begin{proof}[Proof of Theorem \ref{thm_mixed-bistr}]
(3)$\Rightarrow$(2): This is clear. 

(2)$\Rightarrow$(1): By Lemma \ref{lem_pstandard}, we have $\Theta(i) \in \{\Delta(i), \overline{\Delta}(i) \}$.
Hence the assertion follows from $A \in \mathcal{F}(\Theta)$.

(1)$\Rightarrow$(3): This follows from Proposition \ref{prop-BS}. 
\end{proof}

\subsection{Standardization}
Let $\mathcal{C}:=(\mathcal{C},\mathbb{E},\mathfrak{s},\mathbb{E}^{-1})$ be a $\Bbbk$-linear Hom-finite Krull--Schmidt extriangulated category with a negative first extension.
In this subsection, we construct a mixed stratified algebra from a mixed stratifying system in $\mathcal{C}$.
Applying a covariant functor $\Psi$ (respectively, a contravariant functor $\Phi$) to  $\Theta:=(\Theta(1),\Theta(2),\ldots,\Theta(n))$, we obtain  $\Psi(\Theta):=(\Psi(\Theta(1)),\Psi(\Theta(2)),\ldots,\Psi(\Theta(n)))$ (respectively, $\Phi(\Theta):=(\Phi(\Theta(n)),\ldots, \Phi(\Theta(2)), \Phi(\Theta(1)))$).
The following theorem is a main result of this paper.

\begin{theorem}\label{mainthm2}
Let $(\Theta, \mathbb{P})$ be a mixed stratifying system in $\mathcal{C}$ and $B:=\End_{\mathcal{C}}(\mathbb{P})$.
Let $\Psi:=\mathcal{C}(\mathbb{P}, -): \mathcal{\mathcal{C}} \to \mod B$.
Then $(\Psi(\Theta),\Psi(\mathbb{P}))$ is a mixed stratifying system in $\mod B$ and the restriction $\Psi: \mathcal{F}(\Theta) \to \mathcal{F}(\Psi(\Theta))$ is an equivalence of categories.
Moreover, $B$ is a mixed stratified algebra with respect to $\Psi(\Theta)$.
\end{theorem}

To prove Theorem \ref{mainthm2}, we need the following lemma.

\begin{lemma}\label{lem_equiv}
Let $\Lambda:=\Psi(\Theta)$. Then the following statements hold.
\begin{itemize}
\item[(1)] The restriction $\Psi: \mathcal{F}(\Theta)\to \mathcal{F}(\Lambda)$ is an equivalence of categories.
\item[(2)] For each $M, N \in \mathcal{F}(\Theta)$, $\mathbb{E}(M, N) \cong \Ext_{B}^{1}(\Psi(M), \Psi(N))$ holds.
\end{itemize}
\end{lemma}

\begin{proof}
Let $(\Theta, \mathbb{P})$ be a mixed stratifying system in $\mathcal{C}$.

(1) First we show that the restriction $\Psi:\mathcal{F}(\Theta)\rightarrow\mathcal{F}(\Lambda)$ is well-defined.
Let $L\rightarrow M\rightarrow N\dasharrow$ be an $\mathfrak{s}$-conflation in $\mathcal{F}(\Theta)$.
Then it follows from (MSS4) and $\mathbb{P}\in {}^{\perp_{\mathbb{E}}}\mathcal{F}(\Theta)$ that $0\rightarrow \Psi(L)\rightarrow \Psi(M)\rightarrow \Psi(N)\rightarrow 0$ is an exact sequence in $\mod B$.
This property induces that  $\Psi:\mathcal{F}(\Theta)\rightarrow\mathcal{F}(\Lambda)$ is well-defined.

Next we show that $\Psi:\mathcal{F}(\Theta)\to\mathcal{F}(\Lambda)$ is fully faithful.
Let $M,N\in \mathcal{F}(\Theta)$.
By Proposition \ref{prop_mss}(2), there exist $\mathfrak{s}$-conflations
\begin{align}\label{projpres}
K_{1} \to P_{0} \rightarrow M \dashrightarrow,\ K_{2} \to P_{1} \rightarrow K_{1} \dashrightarrow
\end{align}
such that $K_{1}, K_{2} \in \mathcal{F}(\Theta)$ and $P_{0}, P_{1} \in \add \mathbb{P}$.
Applying $\mathcal{C}(-,N)$ to \eqref{projpres}, we obtain exact sequences
\begin{align}
\mathbb{E}^{-1} (K_{1}, N) \to \mathcal{C}(M,N)\to \mathcal{C}(P_{0},N)\to \mathcal{C}(K_{1},N), \notag \\
\mathbb{E}^{-1} (K_{2}, N) \to \mathcal{C}(K_{1},N)\to \mathcal{C}(P_{1},N)\to \mathcal{C}(K_{2},N)\notag
\end{align}
by (NE2).
Since $N, K_{1}, K_{2} \in \mathcal{F}(\Theta)$, it follows from (MSS4) that $\mathbb{E}^{-1}(K_{i}, N)=0$ holds for each $i =1, 2$.
Thus we obtain an exact sequence
\begin{align}\label{leftexseq}
0 \to \mathcal{C}(M,N)\to \mathcal{C}(P_{0},N)\to \mathcal{C}(P_{1},N).
\end{align}
On the other hand, applying $\Psi$ to \eqref{projpres} induces an exact sequence
\begin{align}\label{rightexseq}
\Psi(P_{1})\rightarrow \Psi(P_{0})\rightarrow \Psi(M)\rightarrow 0.
\end{align}
Applying $\Hom_{B}(-,\Psi(N))$ to \eqref{rightexseq} induces an exact sequence
\begin{align}
0\to \Hom_{B}(\Psi(M),\Psi(N))\to \Hom_{B}(\Psi(P_{0}),\Psi(N))\to \Hom_{B}(\Psi(P_{1}),\Psi(N)). \notag
\end{align}
Comparing the exact sequence with \eqref{leftexseq}, we have $\mathcal{C}(M,N)\cong \Hom_{B}(\Psi(M),\Psi(N))$ because $\mathcal{C}(-,N)\to \Hom_{B}(\Psi(-),\Psi(N))$ is a functorial isomorphism on $\add\mathbb{P}$.

Finally we show that $\Psi:\mathcal{F}(\Theta)\to \mathcal{F}(\Lambda)$ is dense.
Let $X\in \mathcal{F}(\Lambda)$.
By Lemma \ref{lem_filt}, we obtain a sequence of monomorphisms
\begin{align}
0=X_{0}\xrightarrow{f_{0}} X_{1}\xrightarrow{f_{1}} X_{2}\xrightarrow{f_{2}} \cdots \xrightarrow{f_{l-1}} X_{l}=X \notag
\end{align}
with $\Lambda_{i}:=\cok f_{i-1} \in \add \Lambda (j_{i})$, where $\Lambda(j_{i}):=\Psi(\Theta(j_{i}))$,  for each $i\in [1,l]$.
By induction on $l$, we show that there exists $M\in \mathcal{F}(\Theta)$ such that $X\cong \Psi(M)$.
If $l=1$, then this is clear by $X_{1}=\Lambda_{1} \in \add\Lambda(j_{1})$.
Assume $l\ge2$.
Then we have the exact sequence $0\to X_{l-1} \xrightarrow{f_{l-1}} X_{l} \to \Lambda_{l} \to 0$.
By $\Lambda_{l} \in \add \Lambda(j_{l})$, there exists $\Theta_{l} \in \add \Theta(j_{l})$ such that $\Psi(\Theta_{l})=\Lambda_{l}$.
Due to (MSS3), we obtain an $\mathfrak{s}$-conflation $K_{l}\xrightarrow{\beta} P_{l}\to \Theta_{l}\dashrightarrow$ with  $K_{l}\in \mathcal{F}(\Theta)$ and $P_{l}\in \add \mathbb{P}$.
Applying $\Psi$ to the $\mathfrak{s}$-conflation induces an exact sequence
\begin{align}
0 \to \Psi(K_{l}) \xrightarrow{\Psi(\beta)} \Psi(P_{l}) \to \Lambda_{l} \to 0. \notag
\end{align}
Since $\Psi(P_{l})$ is a projective $B$-module, we have a commutative diagram of exact sequences
\begin{align}
\xymatrix{
&&0\ar[d]&0\ar[d]&\\
&&\Psi(K_{l})\ar@{=}[r]\ar[d]^-{a}&\Psi(K_{l})\ar[d]^{\Psi(\beta)}&\\
0\ar[r]&X_{l-1}\ar[r]^-{\tiny\left[\begin{smallmatrix}1\\ 0 \end{smallmatrix}\right]}\ar@{=}[d]&X_{l-1}\oplus \Psi(P_{l})\ar[r]^-{\tiny\left[\begin{smallmatrix}0& 1 \end{smallmatrix}\right]}\ar[d]&\Psi(P_{l}) \ar[r]\ar[d]&0&\\
0\ar[r]&X_{l-1}\ar[r] &X_{l}\ar[r]\ar[d]&\Lambda_{l}\ar[r]\ar[d]&0.&\\
&&0&0&
}\notag
\end{align}
By the induction hypothesis, there exists $M_{l-1} \in \mathcal{F}(\Theta)$ such that $\Psi(M_{l-1}) \cong X_{l-1}$.
Since $\Psi:\mathcal{F}(\Theta)\to \mathcal{F}(\Lambda)$ is fully faithful, we have a morphism  $\alpha=\left[\begin{smallmatrix}\alpha_{1}\\\alpha_{2}\end{smallmatrix}\right]: K_{l} \to M_{l-1} \oplus P_{l}$ with $a=\Psi(\alpha)$.
Hence $\alpha_{2}=\beta$ holds.
Since $\beta$ is an $\mathfrak{s}$-inflation, so is $\alpha$ by \cite[Corollary 3.16]{NP19}.
Thus we obtain an $\mathfrak{s}$-conflation $K_{l} \xrightarrow{\alpha} M_{l-1}\oplus P_{l} \to M_{l} \dashrightarrow$ and $\Psi(M_{l})\cong X_{l}$.
We show $M_{l} \in \mathcal{F}(\Theta)$.
By the dual statement of \cite[Proposition 3.17]{NP19}, we have a commutative diagram of $\mathfrak{s}$-conflations
\begin{align}
\xymatrix{
&&K_{l}\ar@{=}[r]\ar[d]^-{\tiny\left[\begin{smallmatrix}\alpha_{1}\\ \beta \end{smallmatrix}\right]}&K_{l}\ar[d]^-{\beta}&\\
&M_{l-1}\ar[r]^-{\tiny\left[\begin{smallmatrix}1\\0 \end{smallmatrix}\right]}\ar@{=}[d]&M_{l-1}\oplus P_{l}\ar[r]^-{\tiny\left[\begin{smallmatrix}0& 1 \end{smallmatrix}\right]}\ar[d]&P_{l}\ar[d]\ar@{-->}[r]&\\
&M_{l-1}\ar[r] &M_{l}\ar[r]\ar@{-->}[d]&\Theta_{l}\ar@{-->}[d]\ar@{-->}[r]&.\\
&&&&
}\notag
\end{align}
In particular, we obtain the $\mathfrak{s}$-conflation $M_{l-1}\rightarrow M_{l}\rightarrow \Theta_{l}\dasharrow$.
Hence $M_{l} \in \mathcal{F}(\Theta)$.

(2) Let $M, N \in \mathcal{F}(\Theta)$.
By Proposition \ref{prop_mss}(2), there exists an $\mathfrak{s}$-conflation $K \to P \to M \dashrightarrow$ such that $K \in \mathcal{F}(\Theta)$ and $P \in \add \mathbb{P}$.
Since $\Psi:\mathcal{F}(\Theta)\to \mathcal{F}(\Lambda)$ is fully faithful, we have a commutative diagram
\begin{align}
\xymatrix{
\mathcal{C}(P,N) \ar[r]\ar[d]^{\cong}& \mathcal{C}(K,N) \ar[r]\ar[d]^{\cong} &\mathbb{E}(M,N) \ar[r] &0\\
\Hom_{B}(\Psi(P), \Psi(N))  \ar[r]& \Hom_{B}(\Psi(K), \Psi(N)) \ar[r]& \Ext_{B}^{1}(\Psi(M), \Psi(N)) \ar[r]& 0.}\notag
\end{align}
Thus we obtain $\mathbb{E}(M, N) \cong \Ext_{B}^{1}(\Psi(M), \Psi(N))$.
This finishes the proof.
\end{proof}

Now we are ready to prove Theorem \ref{mainthm2}.

\begin{proof}[Proof of Theorem \ref{mainthm2}]
We show that $(\Psi(\Theta),\Psi(\mathbb{P}))$ is a mixed stratifying system in $\mod B$.
Let $\Lambda(i):=\Psi(\Theta(i))$ and $\Lambda:=\Psi(\Theta)$.
Since $\mathbb{E}^{-1}(\Lambda, \Lambda)=0$ holds by the negative first extension structure on $\mod B$, we have (MSS4).
By Lemma \ref{lem_equiv}(1), $\Psi$ gives an equivalence $\mathcal{F}(\Theta)\to \mathcal{F}(\Lambda)$.
This implies that if $i>j$, then $\Hom_{B}(\Lambda(i), \Lambda(j))=0$, and hence (MSS2) holds.
Moreover, by Lemma \ref{lem_equiv}, $\Theta(i)$ is a brick (respectively, a stone) if and only if $\Lambda(i)$ is a brick (respectively, a stone).
Thus (MSS1) holds.
Due to (MSS3) for $(\Theta, \mathbb{P})$, there exists an $\mathfrak{s}$-conflation $K(i) \to P(i) \to \Theta(i) \dashrightarrow$ such that $K(i) \in \mathcal{F}(\Theta(\ge i))$ and $P(i) \in \add \mathbb{P}$.
Applying $\Psi$ to the $\mathfrak{s}$-conflation induces an exact sequence $0 \to \Psi(K(i)) \to \Psi(P(i)) \to \Lambda(i) \to 0$ such that $\Psi(K(i)) \in \mathcal{F}(\Lambda(\ge i))$ and $\Psi(P(i))$ is an indecomposable projective $B$-module.
Thus we have (MSS3).

Next we show that $B$ is a mixed stratified algebra with respect to $\Lambda$.
Let $e_{P(i)}\in \End_{\mathcal{C}}(\mathbb{P})$ be a composition map of the projection $\mathbb{P} \to P(i)$ and the injection $P(i) \to \mathbb{P}$.
Then $(e_{P(1)}, \ldots, e_{P{(n)}})$ is a complete ordered set of primitive orthogonal idempotents of $B$.
We regard $B$ as an ordered set $B=(e_{P(1)}B, \ldots, e_{P(n)}B)$.
Thus we have $B\cong\Psi(\mathbb{P})$ as ordered sets.
By Theorem \ref{thm_mixed-bistr}, the assertion holds.
This finishes the proof.
\end{proof}

By Theorem \ref{mainthm2}, we can recover \cite[Proposition 1.3]{ES03}, \cite[Theorem 6.1]{MS16} and \cite[Proposition 7.1]{S19}.
Moreover, by Theorems \ref{mainthm1} and \ref{mainthm2}, we have \cite[Theorem 2]{DR92} and \cite[Theorem 2.3]{ADL08}.

\subsection{Ringel duality}
In this subsection, we study Ringel duality of mixed stratified algebras by using mixed bistratifying systems.
Let $A$ be a mixed stratified algebra with respect to $\Theta:=(\Theta(1),\Theta(2),\ldots,\Theta(n))$.
By Theorem \ref{thm_mixed-bistr}, there exists a mixed bistratifying system $(\Theta, A, \mathbb{T})$.
Moreover, it is known that, if $\Theta=\Delta$ (respectively, $\Theta=\overline{\Delta}$), then $\mathbb{T}$ is a tilting $A$-module (respectively, cotilting $A$-module). 
For the definition of tilting modules, see \cite{M86}. 
In general, $\mathbb{T}$ is neither tilting nor cotilting.
However, $\mathbb{T}$ is a Wakamatsu tilting $A$-module.
We recall the definition of Wakamatsu tilting modules which are a generalization of tilting modules.
See \cite{W88, MR04} for more details.

\begin{definition}[{\cite{W88}}]
Let $A$ be an arbitrary finite dimensional algebra.
An $A$-module $W$ is called a \emph{Wakamatsu tilting module} if it satisfies the following conditions.
\begin{itemize}
\item[(1)] $\Ext_{A}^{i}(W, W)=0$ for each $i \ge 1$.
\item[(2)] There exists an exact sequence
\begin{align}
0 \to A \xrightarrow{f^{0}} W^{0} \xrightarrow{f^{1}} W^{1} \xrightarrow{f^{2}} W^{2} \xrightarrow{f^{3}} \cdots \notag
\end{align}
such that $W^{i} \in \add W$ and $\Ext_{A}^{j}(\cok f^{i},W)=0$ for all $i \ge 0$ and $j\ge 1$.
\end{itemize}
\end{definition}

Dually, we can define Wakamatsu cotilting modules.
However, it is shown in \cite[Proposition 2.2]{BS98} that a Wakamatsu tilting module is left-right symmetric, that is, a module is a Wakamatsu titling module if and only if it is a Wakamatsu cotilting module.
Thus Wakamatsu tilting modules are a common generalization of tilting modules and cotilting modules.

For an $A$-module $X$, let $\pd X$ denote the projective dimension of $X$ and $\id X$ the injective dimension of $X$.

\begin{lemma}\label{lem_wak}
Let $A$ be an arbitrary finite dimensional algebra and let $(\Theta, \mathbb{P}, \mathbb{I})$ be a mixed bistratifying system in $\mod A$. Then the following statements hold.
\begin{itemize}
\item[(1)] If $\mathbb{P}\cong A$, then $\mathbb{I}$ is a Wakamatsu tilting $A$-module. If in addition $\pd \Theta<\infty$, then $\mathbb{I}$ is a tilting $A$-module.
\item[(2)] If $\mathbb{I}\cong DA$, then $\mathbb{P}$ is a Wakamatsu cotilting $A$-module. If in addition $\id \Theta<\infty$, then $\mathbb{P}$ is a cotilting $A$-module.
\end{itemize}
\end{lemma}

\begin{proof}
We only prove (1) since the proof of (2) is similar.
Assume $\mathbb{P}\cong A$.
By Proposition \ref{prop_resol}, $\mathcal{F}(\Theta)$ is a resolving subcategory of $\mod A$.
Then we have
\begin{align}\label{seq_resol_high}
\mathcal{F}(\Theta)^{\perp_{\Ext_{A}^{1}}}=\{\textnormal{$M\in \mod A$ $\mid$ $\Ext_{A}^{j}(\mathcal{F}(\Theta),M)=0$ for all $j\geq 1$}\}.
\end{align}
For the convenience of the readers, we show \eqref{seq_resol_high}, although it is a well-known result.  
Let $M\in \mathcal{F}(\Theta)^{\perp_{\Ext_{A}^{1}}}$ and $X\in\mathcal{F}(\Theta)$.
It is enough to claim $\Ext_{A}^{j}(X,M)=0$ for all $j\ge 2$.
Take a minimal projective resolution of $X$
\begin{align}\label{seq_proj-resol}
\cdots \rightarrow P_{1}\xrightarrow{\rho_{1}} P_{0}\xrightarrow{\rho_{0}}X\to 0.
\end{align}
Since $\mathcal{F}(\Theta)$ is a resolving subcategory, the kernel $\ker \rho_{i}$ is in $\mathcal{F}(\Theta)$ for all $i\ge 0$.
Applying $\Hom_{A}(-,M)$ to \eqref{seq_proj-resol} induces an isomorphism $\Ext_{A}^{j}(X,M)\cong \Ext_{A}^{1}(\ker \rho_{j-2},M)$ for each $j\ge 2$.
By $\ker \rho_{j-2}\in\mathcal{F}(\Theta)$ and $M\in \mathcal{F}(\Theta)^{\perp_{\Ext_{A}^{1}}}$, we have $\Ext_{A}^{j}(X,M)=0$.
Next we show that $\mathbb{I}$ is a Wakamatsu tilting module.
By the dual statement of Proposition \ref{prop_mss}(2), $\mathbb{I}$ is an injective cogenerator of $\mathcal{F}(\Theta)$.
Thus we have $\Ext_{A}^{1}(\mathbb{I},\mathbb{I})=0$ and an exact sequence
\begin{align}\label{seq_add-resol}
0\to A\xrightarrow{f^{0}} I^{0}\xrightarrow{f^{1}} I^{1}\xrightarrow{f^{2}}\cdots
\end{align}
with $I^{i}\in \add \mathbb{I}$, $\cok f^{i}\in \mathcal{F}(\Theta)$ and $\Ext_{A}^{1}(\cok f^{i}, \mathbb{I})=0$ for each $i\ge 0$.
By \eqref{seq_resol_high}, we obtain $\Ext_{A}^{j}(\mathbb{I},\mathbb{I})=0$ and $\Ext_{A}^{j}(\cok f^{i}, \mathbb{I})=0$ for all $j\ge 1$.

In the following, we assume $l:=\pd\Theta<\infty$.
Then we can check $\pd X \le l$ for each $X\in \mathcal{F}(\Theta)$.
In particular, we have $\pd \mathbb{I}\le l$.
Applying $\Hom_{A}(\cok f^{l},-)$ to \eqref{seq_add-resol} induces an isomorphism $\Ext_{A}^{1}(\cok f^{l},\cok f^{l-1})\cong \Ext_{A}^{l+1}(\cok f^{l},A)$.
By $\pd \cok f^{l}\le l$, we have $\Ext_{A}^{1}(\cok f^{l},\cok f^{l-1})=0$.
Thus the exact sequence $0\to \cok f^{l-1}\to I^{l}\to \cok f^{l}\to 0$ splits.
This implies $\cok f^{l-1}\in \add \mathbb{I}$.
Hence $\mathbb{I}$ is a tilting $A$-module.
This finishes the proof.
\end{proof}

Under the assumption of Lemma \ref{lem_wak}, $\mathcal{F}(\Theta)$ is a resolving subcategory with an injective cogenerator.
Thus the statement in Lemma \ref{lem_wak} follows from \cite[Proposition 2.9]{MR04}.

Let $\mathcal{C}:=(\mathcal{C},\mathbb{E},\mathfrak{s},\mathbb{E}^{-1})$ be a $\Bbbk$-linear Hom-finite Krull--Schmidt extriangulated category with a negative first extension and let  $(\Theta, \mathbb{P}, \mathbb{I})$ be a mixed bistratifying system in $\mathcal{C}$.
We put $B:=\End_{\mathcal{C}}(\mathbb{P})$ and $C:=\End_{\mathcal{C}}(\mathbb{I})^{\op}$.
Let $\Psi:=\mathcal{C}(\mathbb{P}, -): \mathcal{C}\to \mod B$ and $\Phi:=\mathcal{C}(-, \mathbb{I}):\mathcal{C}\to\mod C$.
Since $(\Theta, \mathbb{I})$ is a mixed costratifying system in $\mathcal{C}$, a pair $(\Theta^{\op}, \mathbb{I}^{\op})$ is a mixed stratifying system in $\mathcal{C}^{\op}$ (see Remark \ref{rem_op}).
Let $E:=\End_{\mathcal{C}^{\op}}(\mathbb{I}^{\op})$ and $\Upsilon:=\mathcal{C}^{\op}(\mathbb{I}^{\op}, -)$.
Then we obtain $E\cong C$ and $\Upsilon(X^{\op})\cong \mathcal{C}(X, \mathbb{I})=\Phi(X)$ for each $X \in \mathcal{C}$.
Hence it follows from Lemma \ref{lem_equiv} that
$\mathcal{C}(M,N)\cong \Hom_{C}(\Phi(N),\Phi(M))$ and $\mathbb{E}(M,N)\cong \Ext_{C}^{1}(\Phi(N),\Phi(M))$ for each $M,N \in \mathcal{F}(\Theta)$.
The following theorem gives a framework of Ringel duality from the viewpoint of mixed bistratifying systems.

\begin{theorem}\label{thm_rdual}
Let $(\Theta, \mathbb{P}, \mathbb{I})$ be a mixed bistratifying system in $\mathcal{C}$.
Then the following statements hold.
\begin{itemize}
\item[(1)] $(\Phi(\Theta), \Phi(\mathbb{I}), \Phi(\mathbb{P}))$ is a mixed bistratifying system in $\mod C$.
Moreover, $\Phi(\mathbb{P})$ is a Wakamatsu titling $C$-module and $(\Phi'\Phi(\Theta), \Phi'\Phi(\mathbb{P}), \Phi'\Phi(\mathbb{I})) \cong (\Psi(\Theta), \Psi(\mathbb{P}), \Psi(\mathbb{I}))$ holds, where $\Phi':=\Hom_{C}(-,\Phi(\mathbb{P}))$.
In particular, $\End_{\mathcal{C}}(\mathbb{P}) \cong \End_{C}(\Phi(\mathbb{P}))$.
\item[(2)] $(\Psi(\Theta), \Psi(\mathbb{P}), \Psi(\mathbb{I}))$ is a mixed bistratifying system in $\mod B$.
Moreover, $\Psi(\mathbb{I})$ is a Wakamatsu titling $B$-module and $(\Psi'\Psi(\Theta),\Psi'\Psi(\mathbb{I}), \Psi'\Psi(\mathbb{P}))\cong (\Phi(\Theta),\Phi(\mathbb{I}), \Phi(\mathbb{P}))$ holds, where $\Psi':=\Hom_{B}(-,\Psi(\mathbb{I}))$.
In particular, $\End_{\mathcal{C}}(\mathbb{I}) \cong \End_{B}(\Psi(\mathbb{I}))$.
\end{itemize}
\end{theorem}

\begin{proof}
(2) We show that $(\Psi(\Theta), \Psi(\mathbb{P}), \Psi(\mathbb{I}))$ is a mixed bistratifying system in $\mod B$.
By Theorem \ref{mainthm2}, it is enough to claim that $(\Psi(\Theta), \Psi(\mathbb{I}))$ is a mixed costratifying system.
Since (MCS1), (MCS2) and (MCS4) clearly hold, we show (MCS3).
By (MCS3) for $(\Theta, \mathbb{I})$, there exists an $\mathfrak{s}$-conflation $\Theta(i) \to I(i) \to C(i) \dashrightarrow$ such that $I(i) \in \add \mathbb{I}$ and $C(i) \in \mathcal{F}(\Theta(\le i))$.
We put $\Lambda(i):=\Psi(\Theta(i))$.
Since $\mathbb{E}^{-1}(\mathbb{P}, C(i))=0=\mathbb{E}(\mathbb{P}, \Theta(i))$, applying $\Psi$ to the $\mathfrak{s}$-conflation above induces an exact sequence
\begin{align}
 0 \to \Lambda(i) \to \Psi(I(i)) \to \Psi(C(i)) \to 0 \notag
\end{align}
with $\Psi(C(i)) \in \mathcal{F}(\Lambda(\le i))$.
By Lemma \ref{lem_equiv}(2), we have $\Psi(I(i))\in \mathcal{F}(\Lambda)^{\perp_{\Ext_{B}^{1}}}$.
Hence (MCS3) holds.
Thus $(\Psi(\Theta),\Psi(\mathbb{P}),\Psi(\mathbb{I}))$ is a mixed bistratifying system in $\mod B$.
Since $\Psi(\mathbb{P})\cong B$ holds, the module $\Psi(\mathbb{I})$ is a Wakamatsu tilting $B$-module by Lemma \ref{lem_wak}(1).
It follows from Lemma \ref{lem_equiv}(1) that  $\Psi'\Psi(X)=\Hom_{B}(\Psi(X),\Psi(\mathbb{I}))\cong \mathcal{C}(X,\mathbb{I})=\Phi(X)$ for each $X \in\mathcal{F}(\Theta)$.
Hence the assertion holds.

(1) Since $(\Theta^{\op},\mathbb{I}^{\op},\mathbb{P}^{\op})$ is a mixed bistratifying system in $\mathcal{C}^{\op}$, it follows from (2) that $(\Upsilon(\Theta^{\op}),\Upsilon(\mathbb{I}^{\op}),\Upsilon(\mathbb{P}^{\op}))$ is a mixed bistratifying system in $\mod C$, where $\Upsilon:=\mathcal{C}^{\op}(\mathbb{I}^{\op},-)$.
Moreover, $\Upsilon(\mathbb{P}^{\op})$ is a Wakamatsu tilting $E$-module, where $E:=\End_{\mathcal{C}^{\op}}(\mathbb{I}^{\op})$.
By $E\cong C$ and $\Upsilon(X^{\op})\cong\Phi(X)$ for each $X \in\mathcal{C}$, we have the assertion.
\end{proof}

Let $A$ be a basic finite dimensional $\Bbbk$-algebra.
Fix a complete ordered set $(e_{1},e_{2},\ldots, e_{n})$ of primitive orthogonal idempotents of $A$.
We regard $A$ as an ordered set $(e_{1}A, e_{2}A, \ldots, e_{n}A)$.
It is shown in Theorem \ref{thm_mixed-bistr} that, if $A$ is a mixed stratified algebra with respect to $\Theta$, then we have a mixed bistratifying system $(\Theta, A, \mathbb{T})$ in $\mod A$. 
Applying Theorem \ref{thm_rdual} to $(\Theta, A, \mathbb{T})$, we obtain the following result.

\begin{corollary}\cite[Section 4]{BS}\label{cor_bs}
Let $A$ be a mixed stratified algebra with respect to $\Theta$.
Then there exists an $A$-module $\mathbb{T}$ satisfying the following properties. 
\begin{itemize}
\item[(1)] $\mathbb{T}$ is a Wakamatsu tilting $A$-module.
\item[(2)] Let $\Phi:=\Hom_{A}(-,\mathbb{T})$. Then $C:=\End_{A}(\mathbb{T})^{\op}$ is a mixed stratified algebra with respect to $\Phi(\Theta)$ and $\Phi(A)$ is a Wakamatsu tilting $C$-module. 
Moreover, $A\cong\End_{C}(\Phi(A))$ holds. 
\item[(3)] If $\Theta=\Delta$ (i.e., $A$ is a standarly stratified algebra), then $\mathbb{T}$ is a tilting module and $C$ is a standardly stratified algebra.
\end{itemize}
\end{corollary}

\begin{proof}
(1) By Theorem \ref{thm_mixed-bistr}, we obtain a mixed bistratifying system $(\Theta, A, \mathbb{T})$ in $\mod A$. 
Hence Lemma \ref{lem_wak}(1) implies that $\mathbb{T}$ is a Wakamatsu tilting module.

(2) By Theorem \ref{thm_rdual}(1), $(\Phi(\Theta), \Phi(\mathbb{T}), \Phi(A))$ is a mixed bistratifying system in $\mod C$, $\Phi(A)$ is a Wakamatsu tilting $C$-module and $A\cong\End_{C}(\Phi(A))$ holds.  
Since $\Phi(\mathbb{T})\cong C$ as $C$-modules, it follows from Theorem \ref{thm_mixed-bistr} that $C$ is a mixed stratified algebra with respect to $\Phi(\Theta)$. 

(3) Let $\Theta=\Delta$.
By \cite[Proposition 1.8]{AHLU00}, we have $\pd \Delta <\infty$.
Thus it follows from Lemma \ref{lem_wak}(1) that $\mathbb{T}$ is tilting. 
Since $(\Phi(\Theta),C)$ is a mixed stratifying system by the statement (2) and Theorem \ref{thm_mixed-bistr}, it follows from Lemma \ref{lem_equiv} that $\Phi(\Theta(i))$ is a stone for each $i\in [1,n]$.
Hence the assertion follows from Lemma \ref{lem_pstandard}(1).
This completes the proof.
\end{proof}

We give an example of Ringel duality of mixed stratified algebras.

\begin{example}
Assume that $\Bbbk$ is an algebraically closed field.
Let $A$ be the $\Bbbk$-algebra defined by the quiver
\begin{align}
\xymatrix@=15pt{1\ar@(ur,ul)_-{x} \ar@(dr,dl)^-{y} & 2\ar[l]^-{a}\ar@<0.5ex>[r]^-{b} &3 \ar@<0.5ex>[l]^-{c}
}\notag
\end{align}
with relations $bcb, ax, ay, x^{2}, xy, yx, y^{2}$ and $ca$.
Let $\Theta(1)=S(1), \Theta(2)=\cok (e_{3}A \xrightarrow{b \cdot -} e_{2}A)$ and $\Theta(3)=e_{3}A/\rad^{2}e_{3}A$.
Let $T(1)= \ker (\mathsf{D}(Ae_{1}) \xrightarrow{f \neq 0} S(2)), T(2)=\mathsf{D}(Ae_{1})$ and $T(3)=\mathsf{D}(Ae_{2})$, where $\mathsf{D}:=\Hom_{\Bbbk}(-, \Bbbk)$.
Then $(\Theta, A, \mathbb{T})$ is a mixed bistratifying system.
In particular, $\mathbb{T}$ is a Wakamatsu tilting module which is not a tilting module by $\pd \mathbb{T} =\infty$.
Let $\Phi:=\Hom_{A}(-, \mathbb{T})$.
Then $C:=\End_{A}(\mathbb{T})^{\op}$ is isomorphic to the $\Bbbk$-algebra defined by the quiver
\begin{align}
\xymatrix{1\ar@(ul,dl)_-{\alpha}& 2\ar[l]_-{\beta} \ar@/^15pt/[r]^-{\gamma} \ar@/^-15pt/[r]_-{\varphi}&3 \ar[l]_-{\delta}
}\notag
\end{align}
with relations $\alpha^{2}$, $\beta\alpha$, $\delta\beta$, $\gamma\delta\gamma$ and $\varphi\delta\varphi$.
By Corollary \ref{cor_bs}(2), it is a mixed stratified algebra with respect to $\Phi(\Theta)$ and $\End_{C}(\Phi(A)) \cong A$ holds.
\end{example}

\subsection*{Acknowledgements}
We would like to thank Professor Julian K\"ulshammer for informing us about the paper \cite{BS}.

\end{document}